\documentclass[12pt,a4paper]{amsart} 
\usepackage[utf8]{inputenc}
\usepackage[english]{babel}

\usepackage{amsmath}         
\usepackage{amsfonts}
\usepackage{amssymb}
\usepackage{color}

\usepackage{amsthm}

\newtheorem{theorem}{Theorem} [section]
\newtheorem{proposition}[theorem]{Proposition}
\newtheorem{corollary}[theorem]{Corollary} 
\newtheorem{lemma}[theorem]{Lemma}
\newtheorem{definition}[theorem]{Definition}

\newtheorem{conjecture}{Conjecture}

\theoremstyle{remark}
\newtheorem{remark}[theorem]{Remark}

\newcommand{\seg}[1]{\overline{#1}}
\newcommand{\vv}[1]{\overrightarrow{#1}}

\usepackage{graphics}
\usepackage{epsfig}

\usepackage{url}
\usepackage{hyperref}

\usepackage[a4paper,top=2.5cm,bottom=2.8cm,
left=3.2cm,right=3.2cm]{geometry}
\markleft{\hfill \textsc{Perspective Central Triangles Formed from a Triangle and a Transversal} \hfill}
\long\def\void#1{}

\begin{document}
International Journal of  Computer Discovered Mathematics (IJCDM) \\
ISSN 2367-7775 \copyright IJCDM \\
Volume 11, 2026 pp. xx--yy  \\
web: \url{http://www.journal-1.eu/} \\
Received xx July 2026. Published on-line xx MMM 2026\\ 

\copyright The Author(s) This article is published 
with open access.\footnote{This article is distributed under the terms of the Creative Commons Attribution License which permits any use, distribution, and reproduction in any medium, provided the original author(s) and the source are credited.} \\
\bigskip
\bigskip

\begin{center}
	{\Large \textbf{Perspective Central Triangles Formed\\from a Triangle and a Transversal}} \\
	\medskip
	\bigskip
        \bigskip

	\textsc{Stanley Rabinowitz$^a$ and Ercole Suppa$^b$} \\

	$^a$ 545 Elm St Unit 1,  Milford, New Hampshire 03055, USA \\
	e-mail: \href{mailto:stan.rabinowitz@comcast.net}{stan.rabinowitz@comcast.net}\footnote{Corresponding author} \\
	web: \url{http://www.StanleyRabinowitz.com/} \\
	
	$^b$ Via B. Croce 54, 64100 Teramo, Italia \\
	e-mail: \href{mailto:ercolesuppa@gmail.com}{ercolesuppa@gmail.com} \\
	web: \url{https://www.esuppa.it} \\
	
\bigskip

\end{center}
\bigskip

\newcommand{\LL}{\ell }

\textbf{Abstract.}
Let $\LL$ be a line not passing through any vertex of a triangle $ABC$ and not parallel to any side.
Line $\LL$ meets the sidelines $BC$, $CA$, $AB$ of $\triangle ABC$ at points $D$, $E$, $F$,
respectively.
We consider three of the triangles that are formed: $\triangle AEF$, $\triangle BFD$, and $\triangle CDE$.
Placing a fixed triangle center
(such as the incenter, centroid, or orthocenter) in each of these three triangles
determines a \emph{central triangle}. We investigate when the reference triangle and
its central triangle are perspective, i.e., when the lines $AD$, $BE$, and $CF$ are concurrent.

A computer search over the first 1000 centers in the Encyclopedia of Triangle
Centers suggested numerous examples of concurrence. We give elementary geometric
proofs for the circumcenter, orthocenter, and Clawson point, develop a general
criterion for concurrence, and identify several operations, including isogonal
and isotomic conjugation, that preserve this property.

Our main result is a complete characterization of the center functions whose
associated cevians are concurrent for every transversal $\LL$. This yields
an explicit normal form for such centers. We also show
that concurrence depends only on the direction of the transversal, and we
investigate the special case in which the transversal is parallel to the Euler
line.

\bigskip
\textbf{Keywords.} triangle centers, computer-discovered mathematics,
Euclidean geometry, GeometricExplorer.

\medskip
\textbf{Mathematics Subject Classification (2020).} 51M04, 51-08.

\newcommand{\ru}{\rule[-8pt]{0pt}{20pt}}

\newcommand{\red}[1]{\textcolor{red}{#1}}

\newenvironment{code}[2]
{
\medskip
\hspace{#1}%
\begin{minipage}{#2}
\color{blue}
}
{
\color{black}
\smallskip
\end{minipage}%
}

\bigskip
\bigskip
%
\section{Introduction}
\label{section:introduction}

Given a reference triangle $ABC$, suppose three associated triangles are constructed from it.
If a fixed triangle center (such as the incenter, centroid, or orthocenter) is chosen in each of these three triangles,
the resulting three points determine another triangle, called a \emph{central triangle}.

One way to form three triangles is by picking an arbitrary point $D$ in the plane of $\triangle ABC$,
but not on any of its sidelines. Lines are then drawn from point $D$ to the vertices of $\triangle ABC$.
This determines the three \emph{side triangles}, $\triangle DBC$,
$\triangle DCA$, and $\triangle DAB$ as shown in Figure~\ref{fig:D}.
Each side triangle has $D$ as one vertex and two of the vertices of $\triangle ABC$
as its other two vertices.

\begin{figure}[h!t]
\centering
\includegraphics[width=0.9\linewidth]{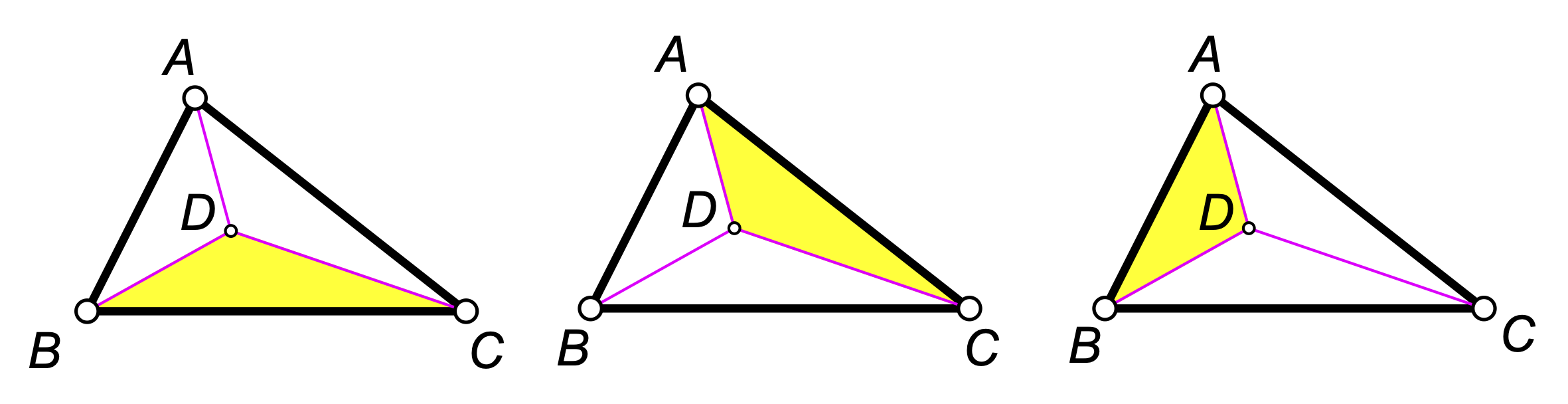}
\caption{Three side triangles determined by $D$}
\label{fig:D}
\end{figure}

Let $X_n$ denote the $n$-th named triangle cataloged in the Encyclopedia of Triangle Centers,
\cite{ETC}.
In \cite{central}, the center $X_n$ (for $n$ ranging from 1 to 1000) was placed in each of the
three side triangles, forming a central triangle.
We investigated the geometric relationship between this central triangle and the reference triangle.
In particular, we determined when the two triangles were similar, homothetic, or perspective.

Another way to form three triangles is to start with a line $\LL$ not passing through
any vertex of $\triangle ABC$ and not parallel to any side. Three points $D$, $E$, and $F$ are determined by
where line $\LL$ intersects the sidelines $BC$, $CA$, and $AB$; we call them the
\emph{pierce points} of $\LL$.
This creates three \emph{corner triangles} associated with the line $\LL$: $\triangle AEF$, $\triangle BFD$, and $\triangle CDE$
as shown in Figure~\ref{fig:L}.
Each corner triangle is bounded by $\LL$ and two of the sidelines of $\triangle ABC$.

\begin{figure}[h!t]
\centering
\includegraphics[width=1.0\linewidth]{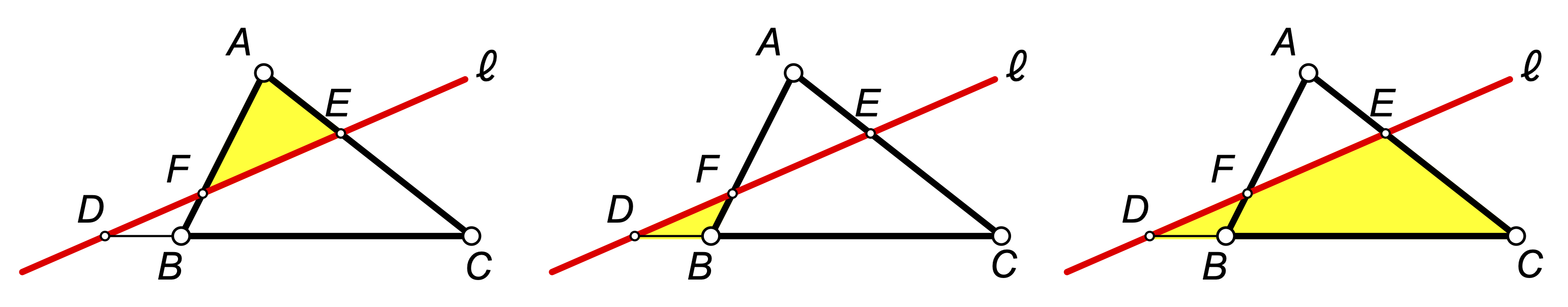}
\caption{Three corner triangles determined by $\LL$}
\label{fig:L}
\end{figure}

In this paper, we will place triangle centers $X_n$ (for $n$ ranging from 1 to 1000)
in each of the three corner triangles forming a central triangle associated with the line $\LL$.
We will determine when this central triangle is perspective with the reference triangle.

\newpage

Figure~\ref{fig:X2} illustrates the construction using the centroid $X_2$.
Here $G$, $H$, and $I$ are the centroids of $\triangle AEF$, $\triangle BFD$, and $\triangle CDE$, respectively.
The red triangle is the resulting central triangle.
In this example the triangles are not perspective. That is, lines $AG$, $BH$, and $CI$ are not concurrent.

\begin{figure}[h!t]
\centering
\includegraphics[width=0.8\linewidth]{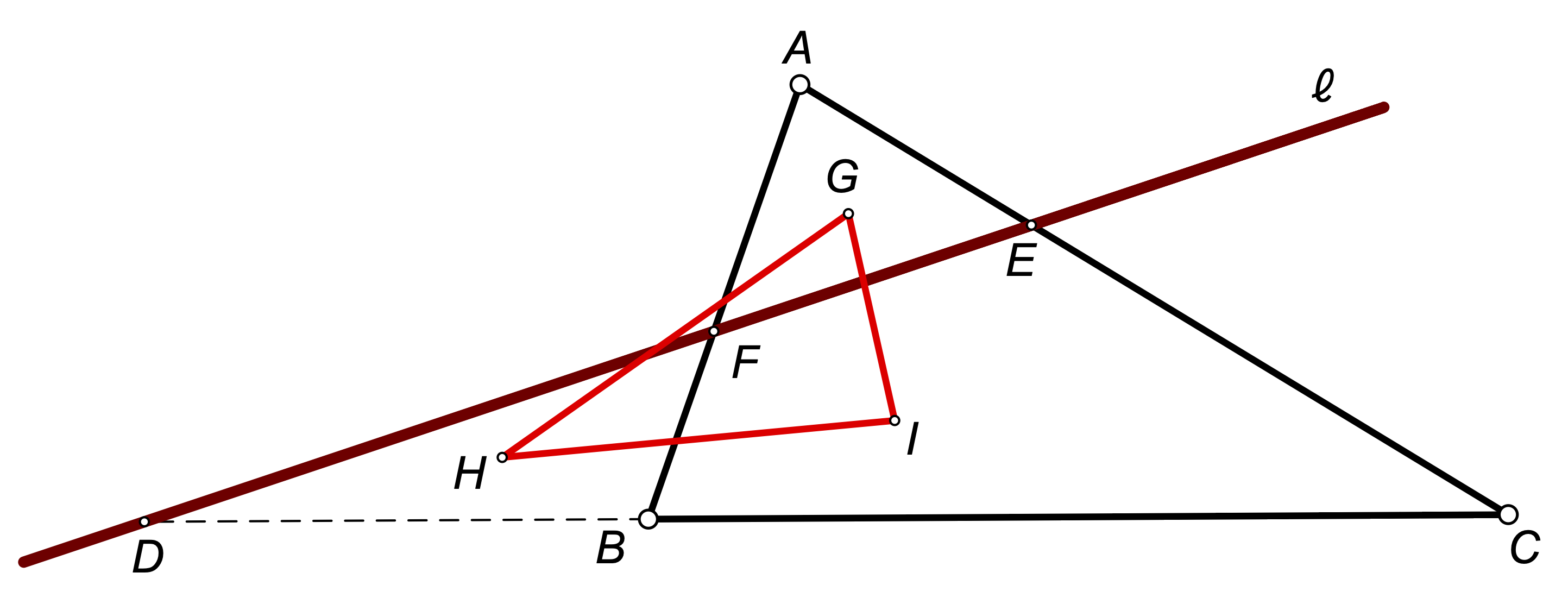}
\caption{Central triangle determined by $\LL$}
\label{fig:X2}
\end{figure}

We used a computer program called GeometricExplorer to compare the reference
triangle with its central triangle.
For an arbitrary triangle $ABC$ and an arbitrary line $\LL$,
the program constructs the three corner triangles formed by $\LL$ and the triangle.
It then places the center $X_n$ in each of these corner triangles, for $1\leq n\leq1000$
(omitting centers at infinity).
The program then determines whether the resulting central triangle is perspective with the reference triangle.

GeometricExplorer uses numerical coordinates (to 15 digits of precision) for locating
all the points. Thus, a relationship found by this program does not constitute a proof that the result is correct,
but gives us compelling evidence for the validity of the result.

The computer search suggested many perspectivities.  The remainder of the paper
establishes these results rigorously, using elementary geometry whenever possible,
and barycentric coordinates otherwise.

We assume the reader is familiar with barycentric coordinates.  For two centers
$P$ and $Q$, we write $P\equiv Q$ to mean that the centers are the same, so that
their barycentric coordinates are proportional.  If their center functions \cite{center} are
$f$ and $g$, we will also write $f\equiv g$.

The paper proceeds as follows.  Section~\ref{section:prelim} develops a general barycentric
criterion for concurrence.  Section~\ref{section:geom} gives elementary geometric proofs for three
representative centers: the circumcenter, the orthocenter, and the Clawson point.
Sections~\ref{sec:parity} and \ref{sec:ops} use these examples to build larger families of concurrent
centers and to identify operations that preserve concurrence.  Section~\ref{sec:universal} gives a
complete characterization of the center functions that are concurrent for every
admissible transversal, including an explicit trigonometric normal form.  The
remaining sections apply these results to the centers found by the computer
search and then study the special case of transversals parallel to the Euler line.

\section{Results}

An \emph{admissible transversal} is a line that does not pass through any vertex of $\triangle ABC$
and is not parallel to any side.

\begin{theorem}
\label{thm:persp}
Let $\LL$ be an arbitrary admissible transversal to $\triangle ABC$.
Three points $D$, $E$, and $F$ are determined by
where line $\LL$ intersects the sidelines $BC$, $CA$, and $AB$, respectively.
Let the $X_n$-points of corner triangles $AEF$, $BFD$, and $CDE$ be $G$, $H$, and $I$ respectively.
If
$$n\in\{3, 4, 19, 24, 25, 48, 49, 63, 68, 69, 92, 93, 184, 186,$$
$$\qquad264, 265, 304, 305, 317, 328, 340, 378, 563, 847\},$$
then triangles $ABC$ and $GHI$ are perspective.
\end{theorem}

Figure~\ref{fig:X19} shows the case where $n=19$.
The red triangle is the central triangle.
The green point is the \emph{perspector} (where $AG$, $BH$, and $CI$ meet).

\begin{figure}[h!t]
\centering
\includegraphics[width=0.7\linewidth]{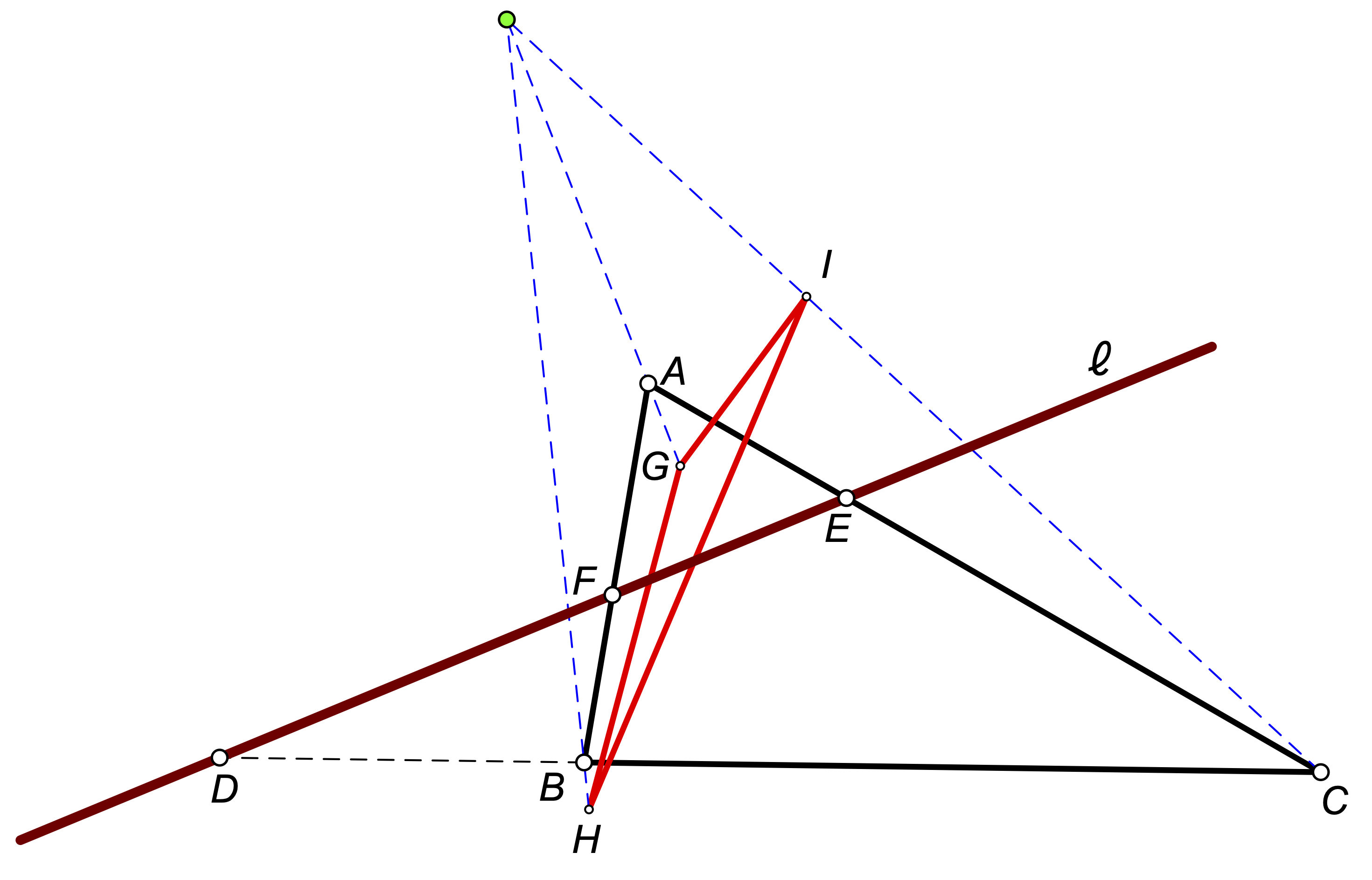}
\caption{Central triangle formed by $\LL$ and $X_{19}$ points}
\label{fig:X19}
\end{figure}

\begin{remark}
For $n\leq 1000$, our computations found no additional centers for which
triangles $ABC$ and $GHI$ are always perspective.
\end{remark}

\section{Preliminary Remarks}
\label{section:prelim}

Before proving Theorem~2.1, we establish a general criterion for deciding when a
triangle center is concurrent.

Let $\LL$ be an admissible transversal of $\triangle ABC$, meeting the sidelines
$BC$, $CA$, and $AB$ at $D$, $E$, and $F$, respectively.
The associated corner triangles are
\[
T_A=AEF,\qquad
T_B=BFD,\qquad
T_C=CDE.
\]

Let $X$ be a triangle center determined by the barycentric center function
$f$.
For a triangle with side lengths
$\lambda_1,\lambda_2,\lambda_3$,
the barycentric coordinates of $X$ are
\begin{equation}\label{eq:cf}
   \bigl(f(\lambda_1,\lambda_2,\lambda_3):f(\lambda_2,\lambda_3,\lambda_1):
   f(\lambda_3,\lambda_1,\lambda_2)\bigr),
\end{equation}
where $f$ is homogeneous and symmetric in its last two arguments.

Let $G,H,$ and $I$ denote the $X$-points of
$T_A,T_B,$ and $T_C$, respectively.
Our goal is to determine when the cevians
$AG$, $BH$, and $CI$ are concurrent.

We say that a center \emph{is concurrent
for $\LL$} if the three lines $AG$, $BH$, $CI$ are concurrent.
If the center $X$ is concurrent for all lines $\LL$, we simply say $X$ is~concurrent.


Throughout the paper, we use homogeneous barycentric coordinates with respect to $\triangle ABC$, writing points
as $(x:y:z)$ and lines as $[u:v:w]$ as in \cite{Yiu}.
Taking $\LL=[u:v:w]$, the admissibility of $\LL$ is equivalent to
$u,v,w$ being nonzero and pairwise distinct.

Using the formula for the intersection of two lines in barycentric coordinates \cite[eq.~(5)]{Grozdev},
we find that the coordinates for points $D$, $E$, and $F$ are
\begin{equation}
\label{eq:DEF}
   D=(0:w:-v),\qquad E=(-w:0:u),\qquad F=(v:-u:0).
\end{equation}
Lengths $\seg{PQ}$ measured along a fixed line are signed, while
$|PQ|$ denotes the unsigned length.
All angle equalities may be read as
equalities of directed angles modulo $180^\circ$, which renders every statement
independent of the position of $\LL$ relative to $\triangle ABC$.

If $P$ has barycentric coordinates $(p:q:r)$, then the \emph{normalized barycentric coordinates} of $P$,
namely $(\frac{p}{p+q+r}:\frac{q}{p+q+r}:\frac{r}{p+q+r})$, are denoted by $\widehat P$.
These have the property that the sum of the three component entries is 1.

\begin{theorem}[Concurrence Condition]
\label{thm:crit}
Order the corner triangles as $T_A=(A,E,F)$, $T_B=(B,F,D)$,
$T_C=(C,D,E)$ and set
\begin{equation}
\label{eq:s}
\begin{aligned}
  s_A&=f(|FA|,|AE|,|EF|), & t_A&=f(|AE|,|EF|,|FA|),\\
  s_B&=f(|DB|,|BF|,|FD|), & t_B&=f(|BF|,|FD|,|DB|),\\
  s_C&=f(|EC|,|CD|,|DE|), & t_C&=f(|CD|,|DE|,|EC|),
\end{aligned}
\end{equation}
where $f$ is a center function.
Write $S=s_As_Bs_C$ and $T=t_At_Bt_C$. Then $AG$, $BH$, $CI$ are concurrent if
and only if $S+T=0$.
\end{theorem}

\begin{proof}
We will first work in barycentric coordinates relative to the corner triangle
$T_A=(A,E,F)$.
By the definition of the center function, the $X$-point, $G$, has coordinates
\[
G=(*:s_A:t_A),
\]
where
\[
s_A=f(|FA|,|AE|,|EF|),\qquad
t_A=f(|AE|,|EF|,|FA|),
\]
with the $A$-coordinate $\ast$ playing no further role.
The cevian $AG$ meets the side $EF$ at
\[
D_A=(0:s_A:t_A).
\]
The corresponding points for $B$ and $C$ are defined similarly.

By the barycentric change of coordinates rule, \cite[eq.~(10)]{Grozdev},
we find that the coordinates of $D_A$ relative to $\triangle ABC$ are $D_A=s_A\widehat E+t_A\widehat F$,
where $\widehat E=(-w,0,u)/(u-w)$, and $\widehat F=(v,-u,0)/(v-u)$.

A line through
$A=(1:0:0)$ has the form $[0:\mu:\nu]$ and passes through $(d_1:d_2:d_3)$ if and only if
$\mu:\nu=d_3:-d_2$.
Reading the last two coordinates of
$D_A$ gives, after clearing the common factor $u$,
\[
   AG=\Bigl[\,0\ :\ \tfrac{s_A}{u-w}\ :\ \tfrac{t_A}{v-u}\,\Bigr],
\]
and cyclically
\[
   BH=\Bigl[\,\tfrac{t_B}{w-v}\ :\ 0\ :\ \tfrac{s_B}{v-u}\,\Bigr],\qquad
   CI=\Bigl[\,\tfrac{s_C}{w-v}\ :\ \tfrac{t_C}{u-w}\ :\ 0\,\Bigr].
\]
The three lines are concurrent precisely when the determinant of their line
coordinates is zero~\cite[eq~(6)]{Grozdev}.
Since each diagonal entry is zero, the determinant reduces to
\[
  \det\!\begin{pmatrix}
      0 & \frac{s_A}{u-w} & \frac{t_A}{v-u}\\[2pt]
      \frac{t_B}{w-v} & 0 & \frac{s_B}{v-u}\\[2pt]
      \frac{s_C}{w-v} & \frac{t_C}{u-w} & 0
  \end{pmatrix}
  =\frac{s_As_Bs_C+t_At_Bt_C}{(u-w)(v-u)(w-v)}
  =\frac{S+T}{(u-w)(v-u)(w-v)} .
\]
The denominator is nonzero because $\LL$ is admissible.
Hence the determinant vanishes exactly when
$S+T=0$.
\end{proof}

\begin{remark}
Theorem~\ref{thm:crit} reduces the geometric problem of concurrence to the algebraic
identity $S+T=0$.
All subsequent proofs amount to showing that this identity holds for particular
center functions.
\end{remark}

\section{Three geometric proofs}
\label{section:geom}

\subsection{The circumcenter $X_3$}

\begin{lemma}\label{lem:isog}
Let $PQR$ have circumcenter $O$. Then $PO$ and the perpendicular from $P$ to $QR$
are isogonal with respect to $\angle QPR$ (Figure~\ref{fig:X3lemma}).
\end{lemma}

\begin{figure}[h!t]
\centering
\includegraphics[width=0.3\linewidth]{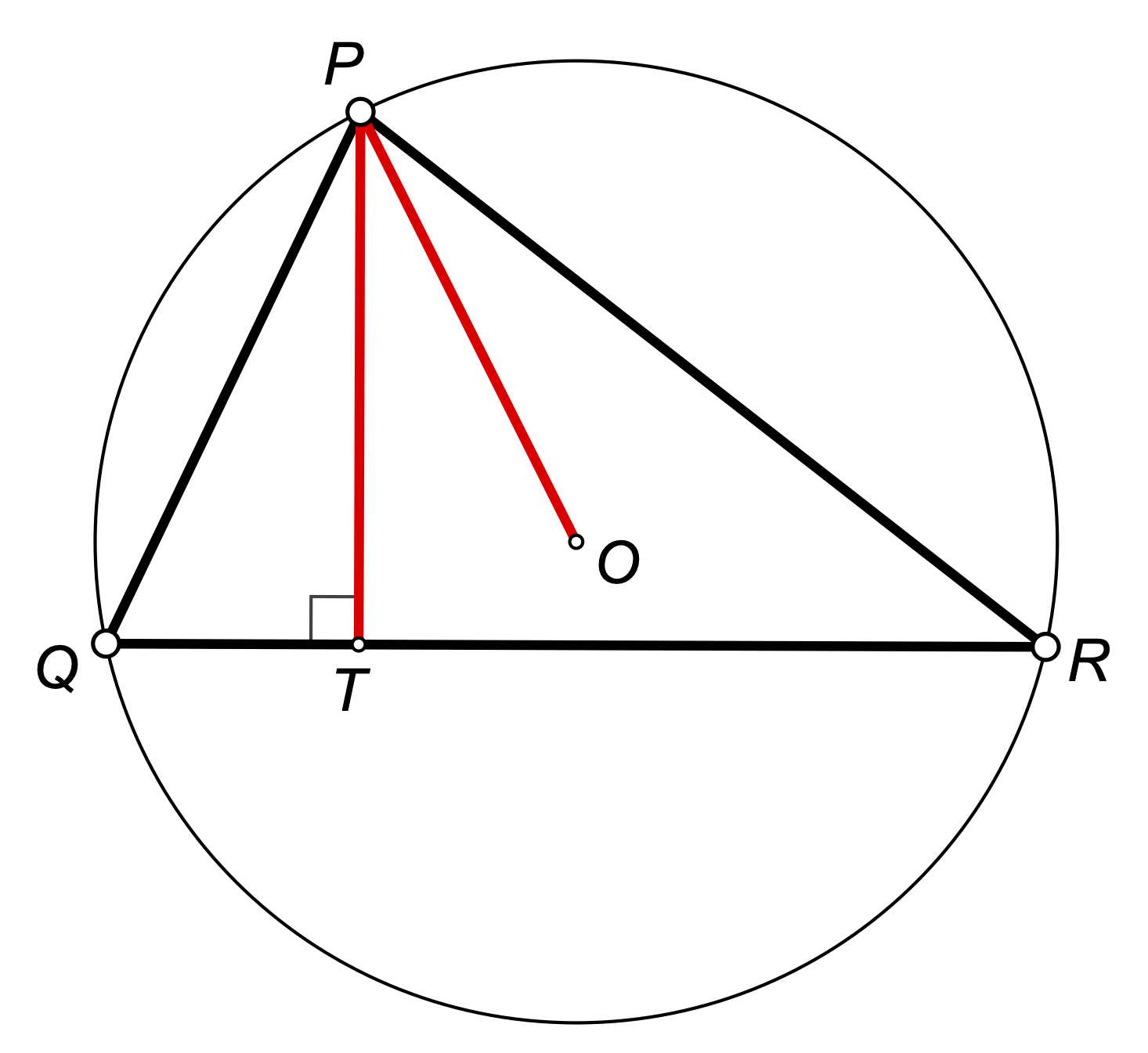}
\caption{Altitude and circumradius of a triangle}
\label{fig:X3lemma}
\end{figure}

\begin{proof}
Let the perpendicular from $P$ meet $QR$ at $T$. In the right triangle $PTQ$ the
angle at $Q$ is $\angle PQR$, so $\angle(PT,PQ)=90^\circ-\angle PQR$.
Since $OP=OR$, triangle $OPR$ is isosceles. Hence
\[
\angle(OP,PR)
=\frac12(180^\circ-\angle POR)
=90^\circ-\angle PQR,
\]
because $\angle POR=2\angle PQR$.
Thus, $\angle(OP,PR)=\angle(QP,PT)$, and $PO$ and $PT$ are reflections in the bisector of $\angle QPR$.
\end{proof}

Since $E$ lies on line $CA$ and $F$ on line $AB$, the sides of $\angle EAF$ are
the lines $CA,AB$, so its bisector is the bisector of $\angle BAC$; and
$EF\subset\LL$. Applying Lemma~\ref{lem:isog} to $T_A=AEF$, and
cyclically, yields the following result.

\begin{corollary}\label{cor:perp3}
The line $AG$ is the isogonal in $\triangle ABC$ of the perpendicular to $\LL$
through $A$; likewise $BH,CI$ are the isogonals of the perpendiculars to $\LL$ through $B,C$.
\end{corollary}

\begin{theorem}
\label{thm:circ}
Center $X_3$ is concurrent. The lines $AG,BH,CI$ concur at a point on the circumcircle of $\triangle ABC$.
\end{theorem}

\begin{figure}[h!t]
\centering
\includegraphics[width=0.55\linewidth]{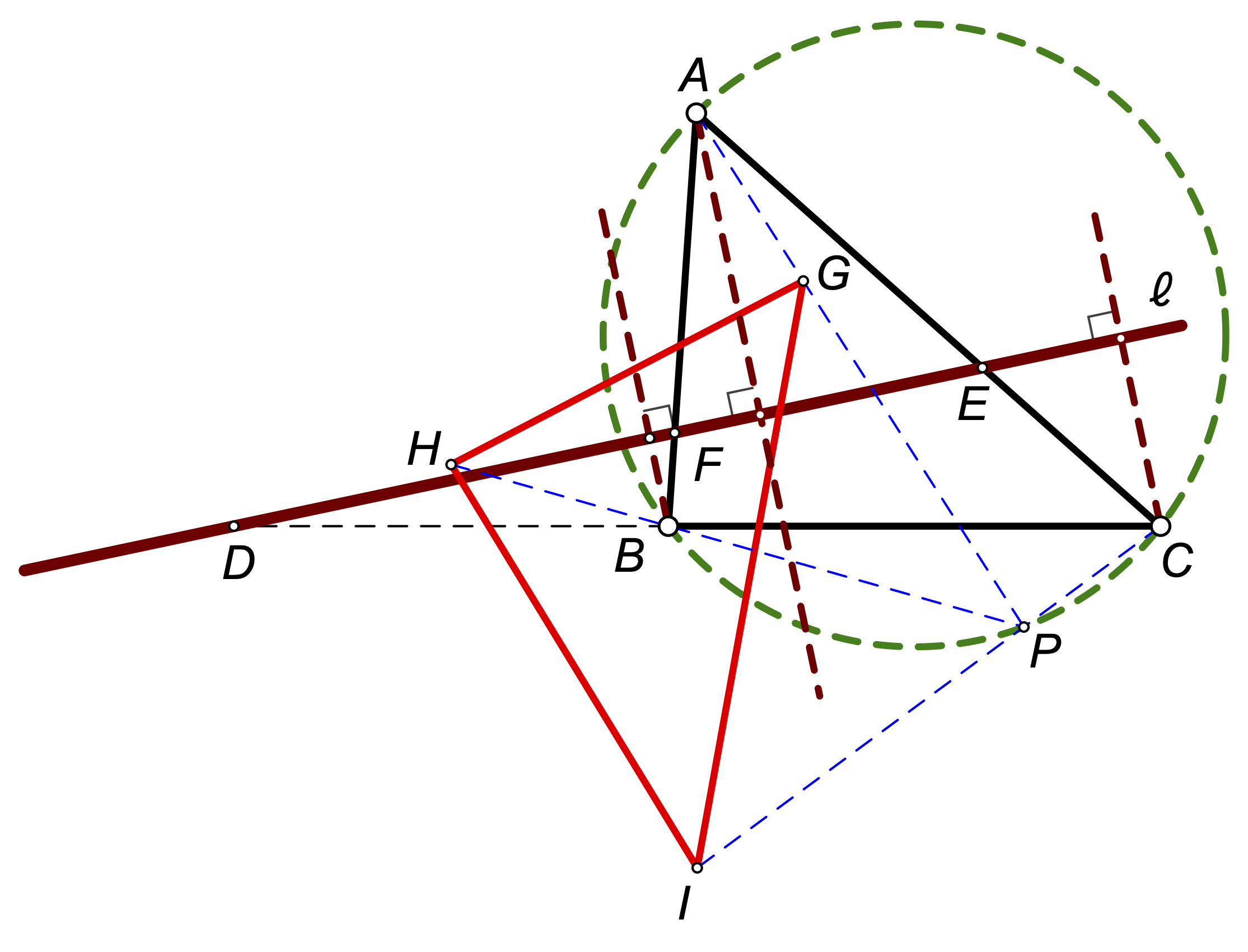}
\caption{$G$, $H$, and $I$ are circumcenters}
\label{fig:X3}
\end{figure}

\begin{proof}
The three perpendiculars to $\LL$ through $A,B,C$ are parallel (Figure~\ref{fig:X3}).
Hence, the
cevians $AP_\infty,BP_\infty,CP_\infty$ of the single point at infinity
$P_\infty$ in the direction orthogonal to $\LL$, are also parallel. By Corollary~\ref{cor:perp3},
reflecting these in the bisectors at $A,B,C$ produces $AG,BH,CI$.
These reflected cevians therefore concur at the isogonal conjugate $P_\infty^*$.
Isogonal conjugation carries the line at infinity to the circumcircle
\cite[\S6]{Yiu}, so $P_\infty^{\ast}$ lies on the circumcircle of $ABC$, and
$AG,BH,CI$ concur there.
\end{proof}

\subsection{The orthocenter $X_4$}\ \\

Assume here that $\triangle ABC$ has no right angle, so that the orthocenter of each corner
triangle is distinct from its apex and its cevians are well defined.

\begin{theorem}
\label{thm:orth}
Center $X_4$ is concurrent. The lines $AG,BH,CI$ are mutually parallel.
Consequently, $\triangle ABC$ and
$\triangle GHI$ are perspective from the point at infinity orthogonal to $\LL$.
\end{theorem}

\begin{figure}[h!t]
\centering
\includegraphics[width=0.5\linewidth]{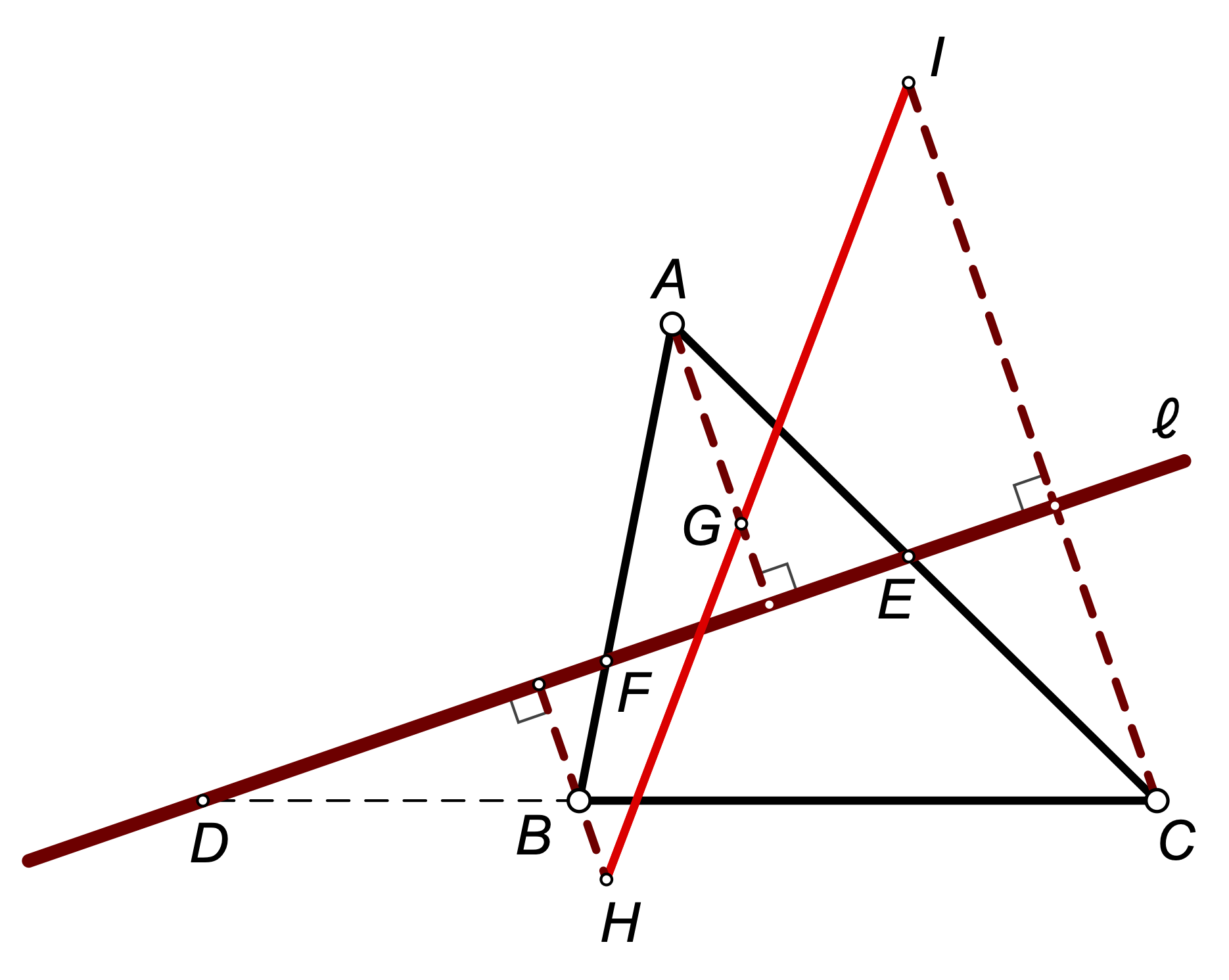}
\caption{$G$, $H$, and $I$ are orthocenters}
\label{fig:X4}
\end{figure}

\begin{proof}
Since $G$ is the orthocenter of $\triangle AEF$, line $AG$ is an altitude and therefore $AG\perp \LL$.
Similarly, $BH\perp\LL$ and $CI\perp\LL$.
Each of $AG,BH,CI$ is perpendicular to the fixed line $\LL$, so the three lines are
parallel and meet at the infinite point $P_\infty$ of the direction orthogonal to $\LL$.
Thus, $\triangle ABC$ and $\triangle GHI$ are perspective from $P_\infty$ (Figure~\ref{fig:X4}).
\end{proof}

\begin{remark}
Theorems~4.3 and~4.4 are closely related.
For each corner triangle, the cevian to the circumcenter is the isogonal of the
cevian to the orthocenter.
Thus, the two perspectors are themselves isogonal conjugates: one lies on the
circumcircle, while the other is the corresponding point at infinity.
\end{remark}


\subsection{The Clawson point $X_{19}$}\ \\

From \cite{ETC}, we find that the Clawson point has center function
\begin{equation}
\label{eq:claw}
   f(x,y,z)=\frac{x}{\,y^{2}+z^{2}-x^{2}\,}.
\end{equation}
Since $D,E,F$ are collinear, Menelaus' theorem \cite[\S3.10]{Altshiller-Court} applies to $\triangle ABC$.

\begin{lemma}[Product Identity]\label{lem:men1}
$\;|AF|\,|BD|\,|CE|=|AE|\,|BF|\,|CD|.$
\end{lemma}

\begin{proof}
The unsigned Menelaus relation for the transversal $\LL$ is
$\frac{|AF|}{|FB|}\cdot\frac{|BD|}{|DC|}\cdot\frac{|CE|}{|EA|}=1$, i.e.
$|AF|\,|BD|\,|CE|=|FB|\,|DC|\,|EA|=|BF|\,|CD|\,|AE|$.
\end{proof}

\begin{lemma}[Signed Ratios]\label{lem:men2}
$\;\dfrac{\seg{EA}}{\seg{EC}}\cdot\dfrac{\seg{FB}}{\seg{FA}}\cdot
\dfrac{\seg{DC}}{\seg{DB}}=1.$
\end{lemma}

\begin{proof}
The signed Menelaus relation is
$\frac{\seg{AF}}{\seg{FB}}\cdot\frac{\seg{BD}}{\seg{DC}}\cdot
\frac{\seg{CE}}{\seg{EA}}=-1$.
Replacing $AF=-FA$, $BD=-DB$, $CE=-EC$ and inverting  yields the asserted relation.
\end{proof}

\goodbreak
\begin{theorem}\label{thm:claw}
Center $X_{19}$ is concurrent.
\end{theorem}

\begin{proof}
Unlike the proofs for $X_3$ and $X_4$, we have not found a purely synthetic
argument for $X_{19}$.
Instead, we apply the concurrence criterion of Theorem~\ref{thm:crit}.

Substituting the center function \eqref{eq:claw} into the concurrence criterion
(Theorem~\ref{thm:crit}) gives
\[
  s_A=\frac{|FA|}{d_A},\ \ t_A=\frac{|AE|}{d_A'},\quad
  s_B=\frac{|DB|}{d_B},\ \ t_B=\frac{|BF|}{d_B'},\quad
  s_C=\frac{|EC|}{d_C},\ \ t_C=\frac{|CD|}{d_C'} .
\]
Hence
\[
  S+T=\frac{|FA|\,|DB|\,|EC|}{d_Ad_Bd_C}
       +\frac{|AE|\,|BF|\,|CD|}{d_A'd_B'd_C'} .
\]
By Lemma~\ref{lem:men1} the two numerators are equal, so proving
$S+T=0$ is equivalent to showing
\begin{equation}\label{eq:DD}
  d_Ad_Bd_C+d_A'd_B'd_C'=0 .
\end{equation}

To prove this identity we express each denominator as a dot product.

By the Law of Cosines, each $d$ is a dot product and we have
\[
\begin{aligned}
 d_A&=2\,\vv{EA}\cdot\vv{EF}, & d_A'&=2\,\vv{FA}\cdot\vv{FE},\\
 d_B&=2\,\vv{FB}\cdot\vv{FD}, & d_B'&=2\,\vv{DB}\cdot\vv{DF},\\
 d_C&=2\,\vv{DC}\cdot\vv{DE}, & d_C'&=2\,\vv{EC}\cdot\vv{ED}.
\end{aligned}
\]

Form the ratio
\[
\frac{d_Ad_Bd_C}
     {d_A'd_B'd_C'}.
\]

Grouping together the factors associated with the same vertex
($E$, $F$, and $D$),
the common cosine factors cancel, leaving
\[
  \frac{d_Ad_Bd_C}{d_A'd_B'd_C'}
  =\underbrace{\frac{\seg{EA}}{\seg{EC}}\cdot\frac{\seg{FB}}{\seg{FA}}\cdot
    \frac{\seg{DC}}{\seg{DB}}}_{\text{side factors}}\;\cdot\;
   \underbrace{\frac{\seg{EF}}{\seg{ED}}\cdot\frac{\seg{FD}}{\seg{FE}}\cdot
    \frac{\seg{DE}}{\seg{DF}}}_{\LL\text{ factors}} .
\]
Each $\LL$ factor equals $-1$ since $\seg{ED}=-\seg{DE}$, and $\seg{FE}=-\seg{EF}$,
$\seg{DF}=-\seg{FD}$. Thus, the product of the three $\LL$ factors is $-1$.
The side factors equal $+1$ by Lemma~\ref{lem:men2}. Therefore the ratio is $-1$,
which is \eqref{eq:DD}.
By Theorem~\ref{thm:crit}, $AG,BH,CI$ concur.
\end{proof}

\begin{remark}
Section~\ref{sec:parity} will show that this phenomenon is not peculiar to $X_{19}$.

\end{remark}

\section{A parity criterion for concurrence}
\label{sec:parity}

The proof of Theorem~\ref{thm:claw} suggests a much larger family of concurrent centers.
The key observation is that the argument depends only on whether the exponent of
the Conway symbol is odd or even.
This leads to the following parity criterion.

Set $S_A=b^{2}+c^{2}-a^{2}$ (twice the Conway symbol),
and $S_B,S_C$ cyclically.
For integers $k$ and $m$, define the center function
\[
   f_{k,m}(x,y,z)=x^{k}\,(y^{2}+z^{2}-x^{2})^{m},
\]
homogeneous in $x,y,z$ and symmetric in $y,z$.
Its center has barycentric coordinates
$$\left(a^{k}S_A^{m}:b^{k}S_B^{m}:c^{k}S_C^{m}\right).$$

\begin{theorem}
\label{thm:parity}
For the center $X$ with center function $f_{k,m}$, the lines $AG,BH,CI$ are concurrent if
and only if $m$ is odd. In particular, the conclusion is independent of $k$.
\end{theorem}

\begin{proof}
With the notation of Theorem~\ref{thm:crit} and the abbreviations
$d_A,d_A',\dots$ of Theorem~\ref{thm:claw},
\[
  s_A=|FA|^{k}d_A^{m},\quad t_A=|AE|^{k}(d_A')^{m},
\]
and cyclically.
Rather than compute $S+T$ directly, we compare the two products by forming the
ratio $S/T$.
We find that
\[
  \frac{S}{T}=\prod_{V}\frac{s_V}{t_V}
  =\left(\frac{|FA|\,|DB|\,|EC|}{|AE|\,|BF|\,|CD|}\right)^{\!k}
   \left(\frac{d_Ad_Bd_C}{d_A'd_B'd_C'}\right)^{\!m}.
\]
The first factor equals $1$ by Lemma~\ref{lem:men1}, for every $k$. The second factor
equals $(-1)^{m}$ by \eqref{eq:DD}. Hence $S=(-1)^{m}T$, so
$S+T=\bigl(1+(-1)^{m}\bigr)T$, which vanishes for every admissible $\LL$ if and
only if $m$ is odd.
The conclusion now follows immediately from Theorem~\ref{thm:crit}.
\end{proof}

\begin{corollary}\label{cor:list}
Among Kimberling centers of index below 100, the following belong to the family
$f_{k,m}$ with $m$ odd, and are therefore concurrent.
\[
\begin{array}{lll}
  X_3=a^{2}S_A=f_{2,1}, & X_4=S_A^{-1}=f_{0,-1}, & X_{19}=aS_A^{-1}=f_{1,-1},\\[2pt]
  X_{25}=a^{2}S_A^{-1}=f_{2,-1}, & X_{48}=a^3S_A=f_{3,1}, & X_{63}=aS_A=f_{1,1},\\[2pt]
  X_{69}=S_A=f_{0,1}. &  & \\[2pt]
\end{array}
\]
No other Kimberling center of index below 100 belongs to this family.
\end{corollary}

\begin{remark}
Theorem~\ref{thm:parity} explains why the proof of Theorem~\ref{thm:claw} worked.
The specific properties of the Clawson point play only a minor role; what matters
is that the Conway symbol appears to an odd power.
Thus, the theorem replaces one isolated concurrence result by an infinite family.
\end{remark}

\section{Operations that preserve concurrence}
\label{sec:ops}

Section~\ref{sec:parity} produced one infinite family of concurrent centers.
We now show that several natural operations on barycentric center functions preserve
concurrence.
Starting from any concurrent center, these operations generate many others.
In particular, they imply that concurrence is preserved under both isogonal and
isotomic conjugation, and they will play a central role in the completeness theorem
of Section~\ref{sec:universal}.

Several simple modifications of a barycentric center function leave the concurrency locus
$\{\,\LL:S+T=0\,\}$ of Theorem~\ref{thm:crit} unchanged. We collect them here.
Together they generate, from a single concurrent center, large families of
concurrent centers; they yield the classical conjugation invariances as a special
case, and they underlie both the sufficient condition of
Theorem~\ref{thm:conc} below and the completeness theorem of~\S\ref{sec:universal}.

The simplest preserving operation is multiplication by a factor that depends only
on the angle opposite the first coordinate.

\begin{lemma}[Multiplier Lemma]
\label{lem:mult}
Let $X_0$ be concurrent with center function $f_0$.
If
\[
f(x,y,z)
=
f_0(x,y,z)
\,
\Phi\!\left(
\frac{(y^2+z^2-x^2)^2}{4y^2z^2}
\right),
\]
where $\Phi$ is any one-variable function,
then $X$ is also concurrent.
\end{lemma}

The argument of $\Phi$ is $\cos^{2}A$,
where $A$ is the angle opposite the first coordinate.

\begin{proof}
The multiplier is symmetric in $y,z$ and homogeneous of degree $0$, so $f$ is a
center function. Let $s_V^{0},t_V^{0}$ be the quantities of Theorem~\ref{thm:crit}
for $f_0$, with products $S_0,T_0$. Each $s_V,t_V$ for $f$ acquires the factor
$\Phi(\cos^{2}\theta)$, where $\theta$ is the angle of the corner triangle
$T_V$ at the vertex opposite the leg placed first. Since that leg
always lies opposite a vertex, these angles are the corner-triangle angles
at $D,E,F$.

At each of $D,E,F$ exactly two corner triangles meet; the two
angles there are both the angle between $\LL$ and the sideline through that
point, equal as undirected lines (vertical or supplementary, according to the
configuration; cf.\ Lemma~\ref{lem:flip}(ii)), so in every case they share the
same $\cos^{2}$ term. Hence, the factor
$\varphi_P:=\Phi(\cos^{2}\theta_P)$ produced at $P\in\{D,E,F\}$ is common to the one
$s$ and the one $t$ that involve $P$, giving
\[
  S=S_0\,\varphi_D\varphi_E\varphi_F,\qquad T=T_0\,\varphi_D\phi_E\varphi_F .
\]
Therefore $S+T=(S_0+T_0)\,\varphi_D\varphi_E\varphi_F=0$, and Theorem~\ref{thm:crit} applies.
\end{proof}

The proof of Theorem~\ref{thm:parity} identifies one infinite family of concurrent centers. We now isolate several simple transformations of barycentric center functions that preserve concurrence. These transformations allow us to generate many new concurrent centers from any known one and will play a central role in the characterization theorem of Section~\ref{sec:universal}.

\begin{definition}\label{def:ops}
For a barycentric center function $f=f(a,b,c)$ write $S_A=b^2+c^2-a^2$ and
$\cos^2\!A=S_A^2/(4b^2c^2)$. The following operations are called \emph{preserving operations}
because they preserve concurrence.
\begin{enumerate}
\item[(O1)] \emph{isotomic} $f\mapsto 1/f$, the center function of the isotomic conjugate~\cite{Yiu},
  defined wherever $f\neq0$;
\item[(O2)] \emph{isogonal} $f\mapsto a^2/f$, the center function of the isogonal conjugate~\cite{Yiu},
  defined wherever $f\neq0$;
\item[(O3)] \emph{side-power} $f\mapsto a^{k}f$ and \emph{co-side-power}
  $f\mapsto (bc)^{j}f$, $\;k,j\in\mathbb Z$;
\item[(O4)] \emph{Conway-even} $f\mapsto S_A^{2m}f$ and $f\mapsto (S_BS_C)^{2n}f$;
\item[(O5)] \emph{symmetric} $f\mapsto \sigma(a,b,c)\,f$ with $\sigma$ symmetric;
\item[(O6)] \emph{angular multiplier} $f\mapsto \Phi(\cos^2\!A)\,f$ for any one-variable function $\Phi$.
\end{enumerate}
\end{definition}

\begin{lemma}
\label{lem:closure}
Each preserving operation in Definition~\ref{def:ops} leaves the concurrency locus
\[
\{\LL: S+T=0\}
\]
of Theorem~\ref{thm:crit} unchanged.
\end{lemma}

\begin{proof}
Recall from Theorem~\ref{thm:crit} that
\[
s_A=f(|FA|,|AE|,|EF|),\qquad
t_A=f(|AE|,|EF|,|FA|),
\]
with analogous definitions for $B$ and $C$, and that
\[
S=s_As_Bs_C,\qquad
T=t_At_Bt_C.
\]

\smallskip
\noindent
\textbf{Operation (O1): isotomic conjugation.}
Replacing $f$ by $1/f$ sends
\[
S\longmapsto \frac1S,\qquad
T\longmapsto \frac1T.
\]
Consequently,
\[
S+T
\longmapsto
\frac1S+\frac1T
=
\frac{S+T}{ST},
\]
which has the same zero locus wherever $S$ and $T$ are finite and nonzero.

\smallskip
\noindent
\textbf{Operation (O2): isogonal conjugation.}
Replacing $f$ by $a^2/f$ is simply the composition of isotomic conjugation with
the side-power operation ($k=2$). Since each preserves the concurrency locus,
their composition does as well.

\smallskip
\noindent
\textbf{Operations (O3)--(O5).}
Each of these operations multiplies every $s_V$ and every $t_V$ by a factor
depending only on the corresponding corner triangle.

For the side-power operation $f\mapsto a^k f$, the products acquire the factors
\[
|FA|^k|DB|^k|EC|^k
\quad\text{and}\quad
|AE|^k|BF|^k|CD|^k,
\]
which are equal by Lemma~\ref{lem:men1}. Thus, both $S$ and $T$ are multiplied by the same
nonzero quantity.

The co-side-power operation is similar. The $\LL$-segment lengths cancel, and
Lemma~\ref{lem:men1} again shows that $S$ and $T$ receive the same factor.

For the Conway-even case, note that $S_A^{2m} = 4^m (bc)^{2m}(\cos^2 A)^m$.
In other words, operation (O4) is just co-side-power (O3) composed with the angular multiplier (O6)
and a symmetric constant.

For the Conway-even, symmetric, and angular operations, each corner triangle
contributes exactly the same multiplier to the corresponding $s$ and $t$. Hence
both $S$ and $T$ are multiplied by the same cyclic product.

In every case,
\[
S+T\longmapsto \Lambda(S+T),
\]
where $\Lambda\neq0$, so the zero locus is unchanged.

\smallskip
\noindent
\textbf{Operation (O6): angular multiplication.}
This has already been proven as the Multiplier Lemma (Lemma~\ref{lem:mult}).
\end{proof}

In other words, we have proven the following.

\begin{theorem}
\label{thm:pres}
Let center $X$ be concurrent for some admissible transversal $\LL$.
Suppose the center function for $X$ is $f$.
Let $X'$  be the center whose center function is obtained from $f$ by applying a preserving operation.
Then $X'$ is concurrent for $\LL$.
\end{theorem}

\begin{corollary}[Conjugation invariance]\label{cor:conj}
Suppose $X$ is concurrent for the admissible line $\LL$ and lies on no sideline
of any corner triangle, so that $S\neq0$ and $T\neq0$. Then its isogonal
conjugate $X^{\ast}$ and its isotomic conjugate $X^{-1}$ are also concurrent
for $\LL$.
\end{corollary}

\begin{proof}
These are the operations (O1) and (O2) of Definition~\ref{def:ops}, shown to preserve the
concurrency locus in Lemma~\ref{lem:closure}; the hypothesis $S,T\neq0$ keeps the
conjugate center functions finite and nonzero.
\end{proof}

\begin{remark}\label{rem:gen}
The monomial family of \S\ref{sec:parity} is visibly closed under (O1) and (O2): isogonal
sends $f_{k,m}$ to $f_{2-k,-m}$ and isotomic sends it to $f_{-k,-m}$, both of
which again have $m$ odd. From the geometric seeds $X_3=f_{2,1}$, $X_4=f_{0,-1}$,
and $X_{19}=f_{1,-1}$ the whole family $\{f_{k,m}:m\ \text{odd}\}$ is generated
(using (O3) and (O4), the seeds $X_3,X_4$ give the even values of $k$, and $X_{19}$ the odd). More
generally, Corollary~\ref{cor:conj} together with
Theorems~\ref{thm:circ}, \ref{thm:orth}, and \ref{thm:claw} shows that every
center obtained from $X_3$, $X_4$, or $X_{19}$ by repeated isogonal and isotomic
conjugation is concurrent---for instance $X_{63},X_{69},X_{264}$, and
$X_{304}$---and this applies equally to concurrent centers outside the monomial
family.
\end{remark}

\begin{remark}
\label{remark:f}
Equivalently, if a concurrent center has first barycentric coordinate of the
form $g(A)$, then the center whose first barycentric coordinate is
$g(A)\,f(\cos^2 A)$ is also concurrent, for any function $f$ of one variable.
\end{remark}

\begin{definition}
An expression that is of the form
\begin{equation}
\label{eq:form1}
\cos A\cdot f(\cos^2A)
\end{equation}
or
\begin{equation}
\label{eq:form2}
\sin A\cos A\cdot f(\cos^2A)
\end{equation}
where $f$ is a function of one variable
is said to be in \emph{concurrent form}.
\end{definition}

The circumcenter and the Clawson point furnish one seed for each form. By
Corollary~\ref{cor:list}, $X_3$ has barycentric center function $a^2S_A$ and
$X_{19}$ has $a/S_A$; using $a\equiv\sin A$ and $S_A\equiv\csc A\cos A=\cot A$ modulo
symmetric factors (the Law of Sines and $S_A=2bc\cos A$), these reduce to
\[
  [X_3]\quad a^2S_A\equiv \sin A\cos A,\qquad
  [X_{19}]\quad a/S_A \equiv \cos A\cdot\frac{1-\cos^2A}{\cos^2A}.
\]
The center function for $X_3$ has been written in the concurrent form \eqref{eq:form2}.
The center function for $X_{19}$ has been written in the concurrent form \eqref{eq:form1}.
Both were
proven concurrent by elementary geometry (Theorems~\ref{thm:circ}
and~\ref{thm:claw}). Multiplying each by an arbitrary $f(\cos^2A)$, as in
Remark~\ref{remark:f}, fills out both forms, so we obtain the following,
giving a broad sufficient condition for concurrence.

\begin{theorem}
\label{thm:conc}
Let $X$ be a center whose first barycentric coordinate is in concurrent form.
Then $X$ is concurrent.
\end{theorem}

The next section shows that Theorem~\ref{thm:conc} is best possible.
Every concurrent center is, up to a symmetric factor, of one of the
forms described above.

\section{Characterization of concurrent centers}
\label{sec:universal}

Sections~\ref{sec:parity} and~\ref{sec:ops} produced broad families of centers that are concurrent for every
admissible transversal.  We now prove that these families are complete.  The goal
of this section is to characterize all center functions whose cevians are
concurrent for every admissible transversal.

The main result is the Characterization Theorem (Theorem~\ref{thm:universal}), followed by the
explicit trigonometric normal form in Theorem~\ref{thm:normalForm}.
\subsection{The shape function and the leg--swap ratio}

We now change the point of view.  Instead of working directly with the
side-length function $f(a,b,c)$, we describe a corner triangle by its angles.

A \emph{shape} will mean an ordered angle triple
\[
\sigma=(\alpha,\beta,\gamma),
\qquad
\alpha+\beta+\gamma=\pi,
\qquad
0<\alpha,\beta,\gamma<\pi.
\]
The angles $\alpha,\beta,\gamma$ are ordered so that they are opposite the first,
second, and third sides of the corner triangle
corresponding to the arguments of the center function $f$, 
as shown in Figure~\ref{fig:cornerAEF}.
We shall
call $\alpha$ and $\beta$ the \emph{base angles} of the shape, and $\gamma$ the
\emph{apex angle}.

\begin{figure}[h!t]
\centering
\includegraphics[width=0.55\linewidth]{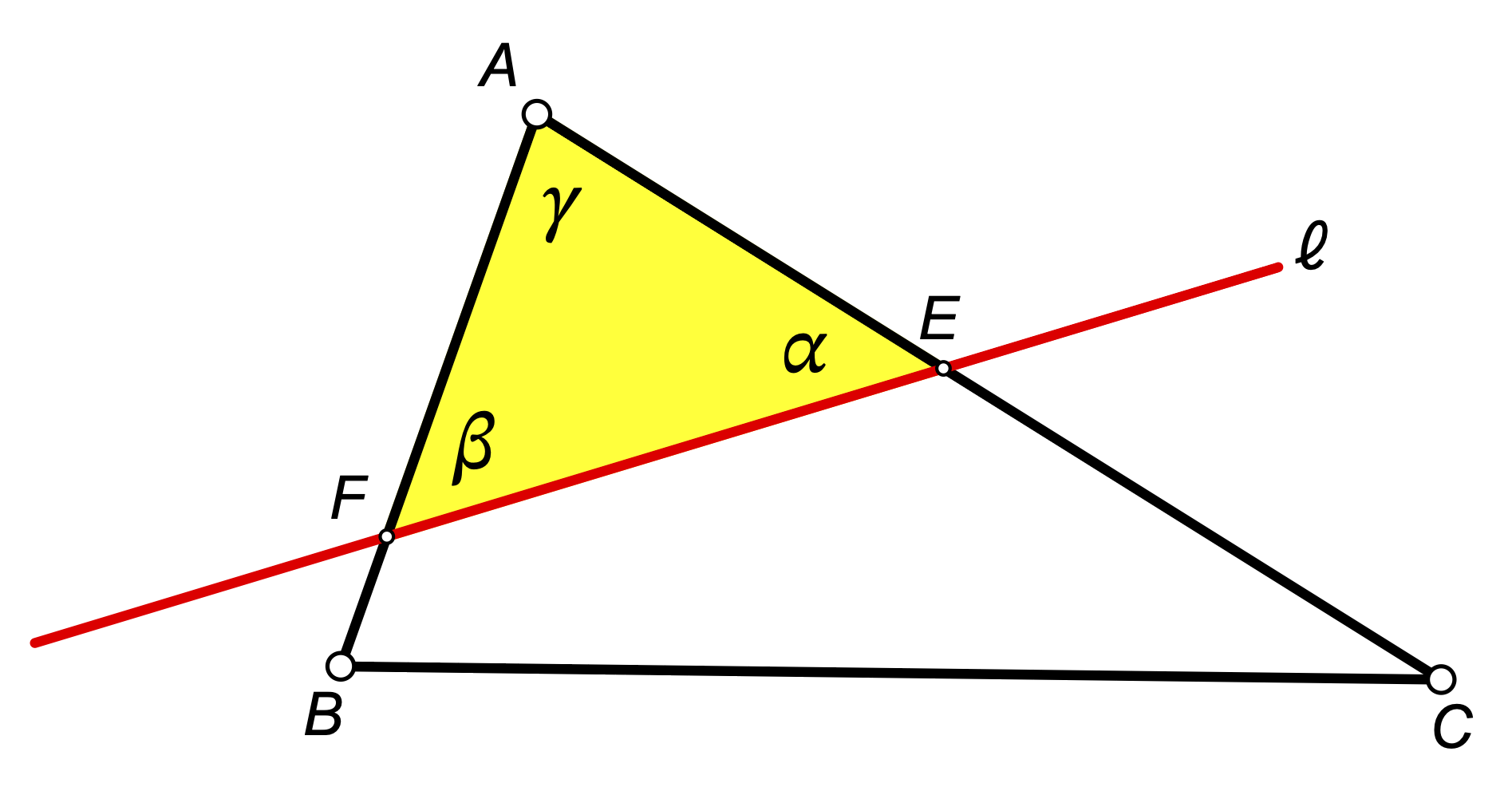}
\caption{The ordered angle triple for the corner triangle $T_A=AEF$.}
\label{fig:cornerAEF}
\end{figure}

Since $f$ is homogeneous, the value of
\[
f(|FA|,|AE|,|EF|)
\]
depends only on the shape of the triangle $AEF$, not on its size.  By the Law of
Sines, the side lengths of a triangle are proportional to the sines of the
opposite angles.  We therefore define the \emph{shape function} associated with
$f$ by
\[
F(\alpha,\beta,\gamma)
:=
f(\sin\alpha,\sin\beta,\sin\gamma).
\]

Since every center function is symmetric in its second and third arguments,
\[
f(p,q,r)=f(p,r,q),
\]
the shape function satisfies
\begin{equation}
\label{eq:legsym}
F(\alpha,\beta,\gamma)=F(\alpha,\gamma,\beta).
\end{equation}

Consider the corner triangle $T_A=AEF$.  The side $FA$ is opposite
$\angle AEF$, the side $AE$ is opposite $\angle AFE$, and the side $EF$ is
opposite $\angle EAF$.  Thus the shape of $T_A$, in the order corresponding to
\[
s_A=f(|FA|,|AE|,|EF|),
\]
is
\[
\sigma_A=(\angle AEF,\angle AFE,\angle EAF).
\]
Similarly, let $\sigma_B$ and $\sigma_C$ denote the corresponding ordered shapes
of the corner triangles $T_B$ and $T_C$.

Using the homogeneity of $f$ and the Law of Sines, we have
$$\frac{t_A}{s_A}=\frac{f(|AE|,|EF|,|FA|)}{f(|FA|,|AE|,|EF|)}=
\frac{f(2R\sin\beta,2R\sin\gamma,2R\sin\alpha)}{f(2R\sin\alpha,2R\sin\beta,2R\sin\gamma)}$$
$$=\frac{f(\sin\beta,\sin\gamma,\sin\alpha)}{f(\sin\alpha,\sin\beta,\sin\gamma)}
=\frac{F(\beta,\gamma,\alpha)}{F(\alpha,\beta,\gamma)}.$$

Using the symmetry of $F$ in its last two variables, this becomes
\[
\frac{t_A}{s_A}
=
\frac{F(\beta,\alpha,\gamma)}
     {F(\alpha,\beta,\gamma)}.
\]

This motivates the following definition.  For any shape
$\sigma=(\alpha,\beta,\gamma)$, define the \emph{leg--swap ratio}
\begin{equation*}
Q(\sigma)
=
Q(\alpha,\beta,\gamma)
:=
\frac{F(\beta,\alpha,\gamma)}
     {F(\alpha,\beta,\gamma)}.
\end{equation*}
Thus, $Q$ measures the effect of interchanging the two base angles while keeping
the apex angle fixed.

Interchanging the two base angles twice returns to the original shape, so
\[
Q(\beta,\alpha,\gamma)=Q(\alpha,\beta,\gamma)^{-1}.
\]

For the three corner triangles, we therefore have
\[
\frac{t_A}{s_A}=Q(\sigma_A),\qquad
\frac{t_B}{s_B}=Q(\sigma_B),\qquad
\frac{t_C}{s_C}=Q(\sigma_C).
\]

The concurrence condition of Theorem~\ref{thm:crit} can now be written in terms of $Q$.
Since
\[
\frac{T}{S}
=
\frac{t_At_Bt_C}{s_As_Bs_C}
=
Q(\sigma_A)Q(\sigma_B)Q(\sigma_C),
\]
the condition $S+T=0$ is equivalent, when $S\neq0$, to
\begin{equation}
\label{eq:crit-Q}
Q(\sigma_A)Q(\sigma_B)Q(\sigma_C)=-1.
\end{equation}

Thus, the problem of characterizing concurrent centers has been reduced to
understanding the single function $Q$.

\subsection{Geometry of the corner triangles}

The three corner triangles are not independent.  Once the admissible transversal
$\LL$ is chosen, the angles of one corner triangle determine those of the other
two.  The following lemma records the elementary geometry of the pierce points
that will be used in the proof of the Characterization Theorem.

Throughout this subsection, $\LL$ is an admissible transversal meeting the
sidelines $BC$, $CA$, and $AB$ at the distinct points $D$, $E$, and $F$.

\begin{definition}
A pierce point is \emph{interior} if it lies in the interior of its side
(for example, $F$ is interior precisely when $A-F-B$), and \emph{exterior}
otherwise.

Exactly one of the three pierce points lies between the other two on the line
$\LL$.  We call this the \emph{median pierce point}; the other two are called
\emph{lateral pierce points}.
\end{definition}

\begin{lemma}[Flip Lemma]
\label{lem:flip}
Let $\LL$ be an admissible transversal.

\begin{enumerate}
\item
At each pierce point the two corner base angles are either equal
or vertical angles or they are supplementary.  They are supplementary precisely when the
pierce point is interior and lateral, or exterior and median.

\item
Exactly one pierce point is supplementary; the other two are equal.

\item
Consequently, exactly one corner triangle has an external apex angle.  It is the
corner triangle that does not contain the supplementary pierce point.  The other
two corner triangles have internal apex angles.

\item
The number of interior pierce points is even.  Hence either none or exactly two
of the pierce points are interior.
\end{enumerate}
\end{lemma}

\begin{proof}
At each pierce point, the two adjacent corner triangles are either equal or determine a pair
of vertical angles; or they are a pair of supplementary angles (Figure~\ref{fig:flip1}).

\begin{figure}[h!t]
\centering
\includegraphics[width=0.7\linewidth]{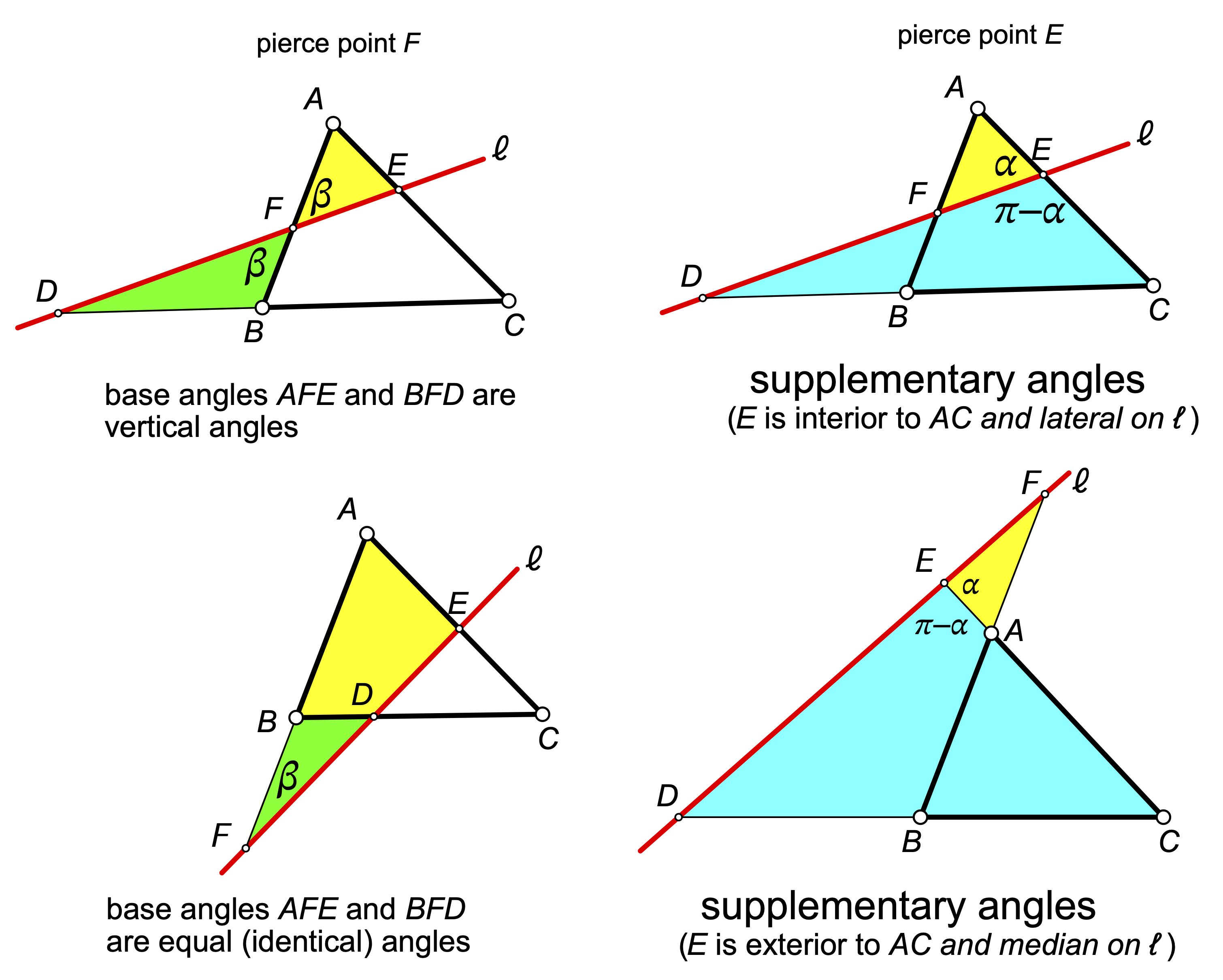}
\caption{case (1)}
\label{fig:flip1}
\end{figure}
The supplementary case
occurs precisely when the pierce point is interior and lateral, or exterior and
median.
This proves (1).

We next show that exactly one pierce point is supplementary.

If none of the pierce points is interior, then all three are exterior (Figure~\ref{fig:flip2}, left).
In this case, the unique median pierce point is exterior and median, so it is
supplementary by (1), while the two lateral pierce points are exterior and lateral, hence equal.
\begin{figure}[h!t]
\centering
\includegraphics[width=0.7\linewidth]{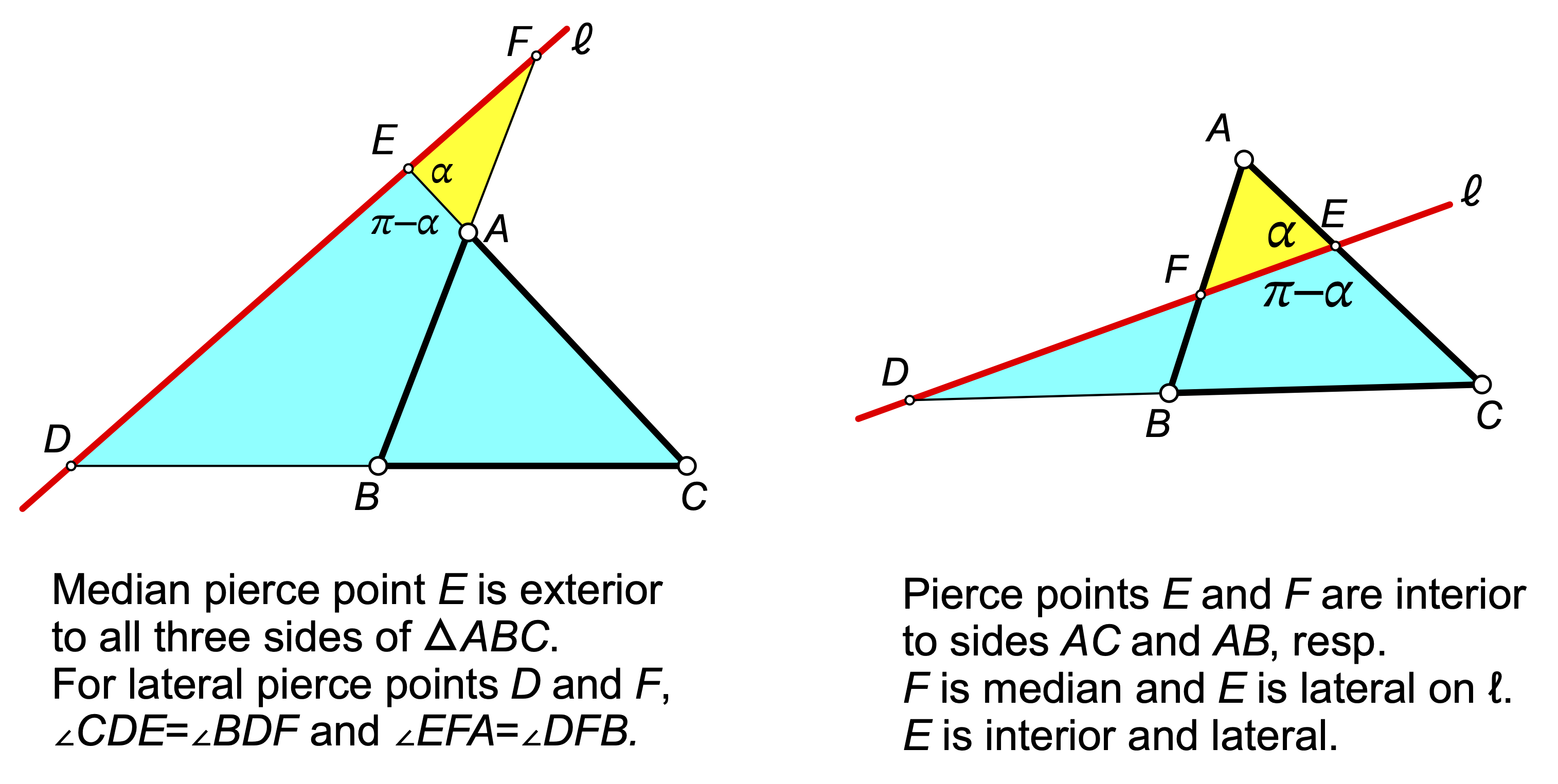}
\caption{case (2)}
\label{fig:flip2}
\end{figure}

If two pierce points are interior, then these two points are the endpoints of the
segment $\LL\cap ABC$.  The third pierce point is exterior and lies outside this
segment (Figure~\ref{fig:flip2}, right).
Hence the median pierce point is one of the two interior points.  Of
the two interior points, one is median and one is lateral.  Thus exactly one
pierce point is interior and lateral, and this is the unique supplementary point.
This proves (2).

The corner triangle with external apex angle is the one opposite the
supplementary pierce point, meaning the corner triangle that does not contain
that pierce point ($\triangle BFD$ in Figure~\ref{fig:flip2}).  This proves (3).

Finally, a line not passing through a vertex can meet the interior of a triangle
only in a segment.  Therefore it enters and leaves the triangle through two sides,
or else it does not enter the triangle at all.  Hence the number of interior
pierce points is either two or zero, proving (4).
\end{proof}

To prove the Characterization Theorem it is convenient to write the three corner
shapes explicitly.  Since the three possible branches differ only by cyclic
permutation of $A$, $B$, and $C$, it suffices to consider one representative case.

Assume that the corner triangle opposite $C$, namely $T_C=CDE$, is the unique
corner triangle with an external apex angle, and set
\[
x=\angle BDF.
\]
Then an angle chase gives
\[
\angle CDE=x,\qquad
\angle CED=C-x,\qquad
\angle DCE=\pi-C.
\]
Similarly for the other two corner triangles.  Writing $C=\pi-A-B$, the three
ordered shapes are
\begin{equation}
\label{eq:shapes}
\begin{aligned}
\sigma_A&=(C-x,\;B+x,\;A),\\
\sigma_B&=(\pi-B-x,\;x,\;B),\\
\sigma_C&=(x,\;C-x,\;\pi-C).
\end{aligned}
\end{equation}
Here the first two entries are the base angles of the shape and the third entry
is the apex angle.

The remaining two branches are obtained by cyclically permuting
$A\rightarrow B\rightarrow C$.  Moreover, the parameters $(A,B,x)$ vary
independently over an open set.  This local freedom will be essential in the
proof of Theorem~\ref{thm:universal}.

\subsection{The Characterization Theorem}
\label{sec:characterization}

Sections \ref{sec:parity} and \ref{sec:ops} produced broad families of centers that are concurrent for
every admissible transversal.  We now prove that these families are complete.

The idea is that a universally concurrent center has only one essential
non-symmetric part.  Up to multiplication by a fully symmetric factor, its shape
function depends only on the first angle of the ordered shape.

We shall use the following terminology.  A function
\[
M(\alpha,\beta,\gamma)
\]
is called \emph{fully symmetric} if it is unchanged by every permutation of
\(\alpha,\beta,\gamma\).

A function \(\phi\) on \((0,\pi)\) is said to be \emph{odd about \(\pi/2\)} if
\[
\phi(\pi-t)=-\phi(t)
\]
for all \(t\in(0,\pi)\).  Thus, \(\phi\) takes opposite values at supplementary
angles.

It is said to be \emph{even about \(\pi/2\)} if
\[
\phi(\pi-t)=\phi(t)
\]
for all \(t\in(0,\pi)\).

Sections \ref{sec:parity} and \ref{sec:ops}  produced broad families of centers that are concurrent for every admissible transversal. We now prove that these families are complete. The central result of this section shows that every universally concurrent center has a remarkably simple structure.

\begin{theorem}[Characterization Theorem]
\label{thm:universal}
A center function \(f\) is concurrent for every admissible transversal if and
only if its shape function factors as
\begin{equation}
\label{eq:family}
F(\alpha,\beta,\gamma)
=
\phi(\alpha)\,M(\alpha,\beta,\gamma),
\end{equation}
where \(M\) is fully symmetric and \(\phi\) is odd about \(\pi/2\).

Equivalently, the leg--swap ratio has the form
\begin{equation}
\label{eq:Qcharacterization}
Q(\alpha,\beta,\gamma)
=
\frac{\phi(\beta)}{\phi(\alpha)},
\end{equation}
where \(\phi\) is odd about \(\pi/2\).
\end{theorem}

Before proving the theorem, we verify that the two formulations are equivalent.

Suppose first that
\[
F(\alpha,\beta,\gamma)
=
\phi(\alpha)\,M(\alpha,\beta,\gamma),
\]
where \(M\) is fully symmetric.  Then
\[
Q(\alpha,\beta,\gamma)
=
\frac{F(\beta,\alpha,\gamma)}
     {F(\alpha,\beta,\gamma)}
=
\frac{\phi(\beta)M(\beta,\alpha,\gamma)}
     {\phi(\alpha)M(\alpha,\beta,\gamma)}
=
\frac{\phi(\beta)}{\phi(\alpha)}.
\]

Conversely, suppose that
\[
Q(\alpha,\beta,\gamma)
=
\frac{\phi(\beta)}{\phi(\alpha)}.
\]
Define
\[
M(\alpha,\beta,\gamma)
=
\frac{F(\alpha,\beta,\gamma)}{\phi(\alpha)}.
\]
Then
\[
\frac{M(\beta,\alpha,\gamma)}
     {M(\alpha,\beta,\gamma)}
=
Q(\alpha,\beta,\gamma)\,
\frac{\phi(\alpha)}{\phi(\beta)}
=
1.
\]
Thus \(M\) is fully symmetric in its first two arguments.  Since every center function
is symmetric in its last two arguments, we also have
\[
F(\alpha,\beta,\gamma)=F(\alpha,\gamma,\beta).
\]
It follows that \(M\) is fully symmetric in its last two arguments as well.  Hence
\(M\) is fully symmetric.

The proof of the Characterization Theorem has two parts.  The sufficiency is
geometric: assuming the factorization
\[
F(\alpha,\beta,\gamma)=\phi(\alpha)M(\alpha,\beta,\gamma),
\]
we show that
\[
Q(\sigma_A)Q(\sigma_B)Q(\sigma_C)=-1
\]
for every admissible transversal.  The necessity is analytic: starting from this
identity for every admissible transversal, we determine the possible form of
\(Q\), and then recover the symmetric factor \(M\).

\begin{proof}[Proof of sufficiency]
Assume that
\[
F(\alpha,\beta,\gamma)=\phi(\alpha)M(\alpha,\beta,\gamma),
\]
where \(M\) is fully symmetric and \(\phi\) is odd about \(\pi/2\).  Then
\[
Q(\alpha,\beta,\gamma)=\frac{\phi(\beta)}{\phi(\alpha)}.
\]

For the three corner triangles, write
\[
\sigma_A=(\alpha_A,\beta_A,\gamma_A),\qquad
\sigma_B=(\alpha_B,\beta_B,\gamma_B),\qquad
\sigma_C=(\alpha_C,\beta_C,\gamma_C),
\]
where the first two entries are the base angles of the corresponding shape.
Then
\[
Q(\sigma_A)Q(\sigma_B)Q(\sigma_C)
=
\frac{\phi(\beta_A)}{\phi(\alpha_A)}
\frac{\phi(\beta_B)}{\phi(\alpha_B)}
\frac{\phi(\beta_C)}{\phi(\alpha_C)}.
\]

Writing these base angles explicitly gives
\[
Q(\sigma_A)Q(\sigma_B)Q(\sigma_C)
=
\frac{\phi(\angle CED)}{\phi(\angle AEF)}
\cdot
\frac{\phi(\angle AFE)}{\phi(\angle BFD)}
\cdot
\frac{\phi(\angle BDF)}{\phi(\angle CDE)}.
\]
Each factor compares the two base angles meeting at one pierce point: the first
at \(E\), the second at \(F\), and the third at \(D\).

By Lemma~\ref{lem:flip}, at two pierce points the two angles are vertical, hence equal, so
the corresponding factors are \(1\).  At the remaining pierce point the two
angles are supplementary, say \(\theta\) and \(\pi-\theta\).  Since \(\phi\) is
odd about \(\pi/2\),
\[
\frac{\phi(\theta)}{\phi(\pi-\theta)}=-1.
\]
Therefore
\[
Q(\sigma_A)Q(\sigma_B)Q(\sigma_C)=-1.
\]
By \eqref{eq:crit-Q}, this is equivalent to \(S+T=0\).  Hence, by
Theorem~\ref{thm:crit}, the cevians \(AG\), \(BH\), and \(CI\) are concurrent.
\end{proof}

\begin{proof}[Proof of necessity]
Assume that the center is concurrent for every admissible transversal.  We shall
show that its shape function has the factorization described in the theorem.

We work on an open region where the relevant functions are nonzero.  This is
the only place where logarithms and signs are used; the resulting identities then
extend to the zero set by continuity.

\emph{Step 1: The functional equation.}

Write
\[
q(u,v)=\log\left|Q(u,v,\pi-u-v)\right|.
\]
Thus \(q\) is the logarithm of the absolute value of the leg--swap ratio, written
as a function of the two base angles.

Taking absolute values in the concurrence condition
\[
Q(\sigma_A)Q(\sigma_B)Q(\sigma_C)=-1
\]
gives
\[
q(\sigma_A)+q(\sigma_B)+q(\sigma_C)=0.
\]

Now use the representative branch from \eqref{eq:shapes}:
\[
\sigma_A=(C-x,\;B+x,\;A),\qquad
\sigma_B=(\pi-B-x,\;x,\;B),\qquad
\sigma_C=(x,\;C-x,\;\pi-C).
\]
Introduce the variables
\[
\lambda=C-x,\qquad \mu=x,\qquad \nu=\pi-B-x.
\]
Then
\[
A=\nu-\lambda,\qquad B=\pi-\nu-\mu,\qquad x=\mu,
\]
and
\[
B+x=\pi-\nu.
\]
Thus the concurrence condition becomes
\begin{equation}
\label{eq:qFunctional}
q(\lambda,\pi-\nu)+q(\nu,\mu)+q(\mu,\lambda)=0.
\end{equation}

\emph{Step 2: Solving the functional equation.}

Differentiate \eqref{eq:qFunctional} with respect to \(\mu\).  The first term is
independent of \(\mu\).  In the second term, \(\mu\) occurs as the second
argument of \(q\); in the third term, \(\mu\) occurs as the first argument of
\(q\).  Hence
\[
\frac{\partial q}{\partial v}(\nu,\mu)
+
\frac{\partial q}{\partial u}(\mu,\lambda)
=
0.
\]
For fixed \(\mu\), the first term is independent of \(\lambda\), while the
second term is independent of \(\nu\).  Since \(\lambda,\mu,\nu\) vary
independently over an open set, each term must depend only on \(\mu\).  In
particular,
\[
\frac{\partial q}{\partial v}(u,v)
\]
depends only on \(v\).

It follows that
\[
q(u,v)=m(u)+K(v)
\]
for some one-variable functions \(m\) and \(K\).

On the other hand, interchanging the two base angles inverts \(Q\), so
\[
Q(v,u,\pi-u-v)=Q(u,v,\pi-u-v)^{-1}.
\]
Therefore
\[
q(v,u)=-q(u,v).
\]
In particular \(q(u,u)=0\), and hence
\[
m(u)+K(u)=0.
\]
Thus \(m=-K\), and so
\[
q(u,v)=K(v)-K(u).
\]

Put
\[
\Phi(t)=e^{K(t)}.
\]
Then
\begin{equation}
\label{eq:absQ}
\left|Q(\alpha,\beta,\gamma)\right|
=
\frac{\Phi(\beta)}{\Phi(\alpha)}.
\end{equation}

We shall also need one consequence of \eqref{eq:qFunctional}.  Setting
\(\mu=\lambda\) gives
\[
q(\lambda,\pi-\nu)+q(\nu,\lambda)=0.
\]
Using \(q(u,v)=K(v)-K(u)\), this becomes
\[
K(\pi-\nu)-K(\nu)=0.
\]
Hence
\[
\Phi(\pi-t)=\Phi(t).
\]
Thus \(\Phi\) is even about \(\pi/2\).

\emph{Step 3: Determining the sign.}

Equation \eqref{eq:absQ} determines the absolute value of \(Q\).  It remains to
determine its sign.

Write
\[
Q(u,v,\pi-u-v)
=
\varepsilon(u,v)\frac{\Phi(v)}{\Phi(u)},
\]
where
\[
\varepsilon(u,v)\in\{1,-1\}.
\]
Since interchanging \(u\) and \(v\) inverts \(Q\), we have
\[
\varepsilon(v,u)=\varepsilon(u,v).
\]
Also,
\[
Q(u,u,\pi-2u)=1,
\]
so
\[
\varepsilon(u,u)=1.
\]

Substituting the above expression for \(Q\) into the concurrence condition, and
using the fact that the absolute values already multiply to \(1\), gives
\begin{equation}
\label{eq:signEquation}
\varepsilon(\lambda,\pi-\nu)\,
\varepsilon(\nu,\mu)\,
\varepsilon(\mu,\lambda)
=
-1.
\end{equation}

Set \(\mu=\lambda\) in \eqref{eq:signEquation}.  Since
\(\varepsilon(\lambda,\lambda)=1\), we get
\[
\varepsilon(\lambda,\pi-\nu)\,\varepsilon(\nu,\lambda)=-1.
\]
Using the symmetry of \(\varepsilon\), this becomes
\begin{equation}
\label{eq:signSupplementary}
\varepsilon(\lambda,\pi-\nu)=-\varepsilon(\lambda,\nu).
\end{equation}

Substituting \eqref{eq:signSupplementary} back into
\eqref{eq:signEquation} gives
\[
\varepsilon(\lambda,\nu)\,
\varepsilon(\nu,\mu)\,
\varepsilon(\mu,\lambda)
=
1.
\]

Choose a fixed angle \(r\) in the interval under consideration, and define
\[
\chi(t)=\varepsilon(t,r).
\]
Taking \(\mu=r\) in the last identity gives
\[
\varepsilon(\lambda,\nu)\,\varepsilon(\nu,r)\,\varepsilon(r,\lambda)=1.
\]
Therefore
\[
\varepsilon(\lambda,\nu)=\chi(\lambda)\chi(\nu).
\]

Also, putting \(\lambda=r\) in \eqref{eq:signSupplementary} gives
\[
\chi(\pi-\nu)=-\chi(\nu).
\]
Thus \(\chi\) is odd about \(\pi/2\).

\emph{Step 4: Recovering \(\phi\) and \(M\).}

Define
\[
\phi(t)=\chi(t)\Phi(t).
\]
Since \(\Phi\) is even about \(\pi/2\) and \(\chi\) is odd about \(\pi/2\), the
function \(\phi\) is odd about \(\pi/2\).

Moreover,
\[
Q(\alpha,\beta,\gamma)
=
\varepsilon(\alpha,\beta)\frac{\Phi(\beta)}{\Phi(\alpha)}
=
\chi(\alpha)\chi(\beta)\frac{\Phi(\beta)}{\Phi(\alpha)}
=
\frac{\phi(\beta)}{\phi(\alpha)}.
\]
By the equivalence proved before the theorem, this implies that
\[
F(\alpha,\beta,\gamma)=\phi(\alpha)M(\alpha,\beta,\gamma),
\]
where \(M\) is fully symmetric.

This proves the necessity and completes the proof of the Characterization
Theorem.
\end{proof}

\subsection{The complete family of concurrent centers}

Theorems~\ref{thm:conc} and~\ref{thm:universal} together give an explicit description of every center
function whose cevians are concurrent for every admissible transversal.

\begin{theorem}[The Trigonometric Normal Form]
\label{thm:normalForm}
A center function is concurrent for every admissible transversal if and only if,
up to multiplication by a symmetric factor,
\[
f(a,b,c)
=
\cos A\,\Phi(\cos^2A)
+
\sin A\cos A\,\Psi(\cos^2A),
\]
where $A$ is the angle opposite the first argument and $\Phi,\Psi$ are functions of one variable.

Thus, the two concurrent forms of Theorem~\ref{thm:conc} are not merely sufficient; together
they exhaust all centers concurrent for every admissible transversal.
\end{theorem}

\begin{proof}
By the Characterization Theorem, a concurrent center has shape function
\[
F(\alpha,\beta,\gamma)=\phi(\alpha)M(\alpha,\beta,\gamma),
\]
where \(M\) is fully symmetric and \(\phi\) is odd about \(\pi/2\), i.e. $\phi(\pi-t)=-\phi(t)$.  Since the
symmetric factor \(M\) may be absorbed into the phrase ``up to multiplication by
a symmetric factor,'' it remains to describe the possible form of the one-variable
factor \(\phi(A)\).

Assume first that the center is concurrent.  Then
\[
\phi(\pi-A)=-\phi(A).
\]
Away from \(A=\pi/2\), define
\[
h(A)=\frac{\phi(A)}{\cos A}.
\]
Since
\[
\cos(\pi-A)=-\cos A,
\]
we have
\[
h(\pi-A)
=
\frac{\phi(\pi-A)}{\cos(\pi-A)}
=
\frac{-\phi(A)}{-\cos A}
=
h(A).
\]
Thus, \(h\) is even about \(\pi/2\).

We now write \(h\) as a function of \(\sin A\).  For \(0<s<1\), define
\[
H(s)=h(\arcsin s)
\]
and assume that $H(s)$ extends continuously to $s=1$.
If \(0<A<\pi\) and \(s=\sin A\), then either
\[
\arcsin(\sin A)=A
\]
or
\[
\arcsin(\sin A)=\pi-A.
\]
In the second case, \(h(\pi-A)=h(A)\).  Therefore, in all cases,
\[
h(A)=H(\sin A).
\]
Hence
\[
\phi(A)=\cos A\,H(\sin A).
\]

It remains to rewrite \(H(\sin A)\) in the stated form. Put $s=\sin A$.
Since \(0<A<\pi\), we have \(s>0\).


Decompose $H$ into its even and odd parts.
That is, write
$H(s)=H_e(s)+H_o(s)$,
where
$H_e$ is even and $H_o$ is odd.
Since \(H_e\) is even, it depends only
on \(s^2\), so
\[
H_e(s)=P(s^2)
\]
for some one-variable function \(P\).  Since \(H_o\) is odd, \(H_o(s)/s\) is even,
so
\[
H_o(s)=sQ(s^2)
\]
for some one-variable function \(Q\).  Hence
\[
H(s)=P(s^2)+sQ(s^2).
\]

Since
\[
s^2=\sin^2 A=1-\cos^2 A,
\]
we may rewrite this as
\[
H(\sin A)
=
\Phi(\cos^2 A)+\sin A\,\Psi(\cos^2 A)
\]
for suitable one-variable functions \(\Phi\) and \(\Psi\).
Note that \(\Phi\) and \(\Psi\) are arbitrary functions.
Therefore
\[
\phi(A)
=
\cos A\,\Phi(\cos^2 A)
+
\sin A\cos A\,\Psi(\cos^2 A),
\]
as claimed.
This proves the necessity.

Conversely, suppose that, up to multiplication by a symmetric factor, $f$ has the form
\[
\cos A\,\Phi(\cos^2 A)
+
\sin A\cos A\,\Psi(\cos^2 A).
\]
When \(A\) is replaced by \(\pi-A\), the factor \(\cos A\) changes sign, while
\(\sin A\) and \(\cos^2 A\) remain unchanged.  Therefore, the whole expression
changes sign:
\[
\phi(\pi-A)=-\phi(A).
\]
Thus, \(\phi\) is odd about \(\pi/2\).  By the Characterization Theorem, the center
is concurrent for every admissible transversal.
\end{proof}

\section{Proof of Theorem~\ref{thm:persp}}\label{sec:proof}

The characterization obtained in Section~\ref{sec:universal} now allows us to explain every
concurrent center found in the original computer search.

We are now at a point where we can give a proof of Theorem~\ref{thm:persp}.

\begin{proof}
The table below lists, for each center in Theorem~\ref{thm:persp}, its barycentric center function.
For these centers, it is a pure function of the angle $A$. Such centers are known as \emph{major centers}.
For each center, the table lists

(1) the ETC barycentric center function,

(2) an equivalent concurrent form, and

(3) the preserving operations used to obtain it.

Column 2 records the center function obtained from \cite{ETC}. Since ETC often lists
only trilinear coordinates, we sometimes take the first trilinear coordinate and
multiply by $a$ to convert it to a barycentric coordinate. Whenever this leaves a
factor of $a$ or of $bc$, we remove it by means of
\[
   a\equiv\sin A,\qquad bc\equiv\csc A.
\]
Each replacement alters the center function only by a factor symmetric in $a,b,c$,
and a symmetric factor multiplies all three barycentric coordinates equally; it
therefore leaves the center unchanged, and preserves the concurrency locus by
Lemma~\ref{lem:closure}. Indeed, by the Law of Sines, $a=2R\sin A$, where the
circumradius $R$ is symmetric, so $a$ and $\sin A$ differ by the symmetric factor
$2R$; and from the area formula $2K=bc\sin A$, we get $bc=2K\csc A$, where the area
$K$ is symmetric, so $bc$ and $\csc A$ differ by the symmetric factor $2K$. After
these reductions every entry of column 2 is a function of $A$ alone.

Column 3 lists an equivalent center function in concurrent form, obtained from
column 2 by standard trigonometric identities
together with the discarding of nonzero constant factors. Each column-3 entry is
$\cos A$ or $\sin A\cos A$ times a function of $\cos^2A$, hence is in concurrent form.
By Theorem~\ref{thm:conc}, the corresponding center is concurrent. As columns
2 and 3 define the same center, every center listed in Theorem~\ref{thm:persp} is
concurrent.
\end{proof}

\newpage

\begin{center}
\begin{tabular}{|l|l|l|}
\hline
\multicolumn{3}{|c|}{\textbf{\color{blue}\large \strut Barycentric Center Functions}}\\ \hline
\textbf{center}&\textbf{ETC form}&\textbf{concurrent form}\\ \hline
 $X_{3}$ & $\sin 2A$ & $\sin A\cos A$ \\ \hline
 $X_{4}$ & $\tan A$ & $\sin A\cos A/\cos^2A$ \\ \hline
 $X_{19}$ & $\sin A\tan A$ & $(\cos A)(1-\cos^2A)/\cos^2A$ \\ \hline
 $X_{24}$ & $\tan A\cos 2A$ & $(\sin A\cos A)(2\cos^2A-1)/\cos^2A$ \\ \hline
 $X_{25}$ & $\sin^2 A\tan A$ & $(\sin A\cos A)(1-\cos^2A)/\cos^2A$ \\ \hline
 $X_{48}$ & $\sin A\sin 2A$ & $(\cos A)(1-\cos^2A)$ \\ \hline
 $X_{49}$ & $\sin A\cos 3A$ & $(\sin A\cos A)(4\cos^2A-3)$ \\ \hline
 $X_{63}$ & $\cos A$ & $\cos A$ \\ \hline
 $X_{68}$ & $\tan 2A$ & $(\sin A\cos A)/(2\cos^2A-1)$ \\ \hline
 $X_{69}$ & $\cot A$ & $(\sin A\cos A)/(1-\cos^2A)$ \\ \hline
 $X_{92}$ & $\sec A$ & $(\cos A)/\cos^2A$ \\ \hline
 $X_{93}$ & $\sin A\sec 3A$ & $(\sin A\cos A)/\bigl((\cos^2A)(4\cos^2A-3)\bigr)$ \\ \hline
 $X_{184}$ & $\sin^3 A\cos A$ & $(\sin A\cos A)(1-\cos^2A)$ \\ \hline
 $X_{186}$ & $\sin A\sin 3A\csc 2A$ & $(\sin A\cos A)(4\cos^2A-1)/\cos^2A$ \\ \hline
 $X_{264}$ & $\csc 2A$ & $(\sin A\cos A)/\bigl((1-\cos^2A)(\cos^2A)\bigr)$ \\ \hline
 $X_{265}$ & $\sin A\sin 2A\csc 3A$ & $(\sin A\cos A)/(4\cos^2A-1)$ \\ \hline
 $X_{304}$ & $\cos A\csc^2 A$ & $(\cos A)/(1-\cos^2A)$ \\ \hline
 $X_{305}$ & $\csc^3 A\cos A$ & $(\sin A\cos A)/(1-\cos^2A)^2$ \\ \hline
 $X_{317}$ & $\cot 2A$ & $(\sin A\cos A)(2\cos^2A-1)/\bigl((1-\cos^2A)(\cos^2A)\bigr)$ \\ \hline
 $X_{328}$ & $\cos A\csc 3A$ & $(\sin A\cos A)/\bigl((1-\cos^2A)(4\cos^2A-1)\bigr)$ \\ \hline
 $X_{340}$ & $\sec A\sin 3A\csc^2 A$ & $(\sin A\cos A)(4\cos^2A-1)/\bigl((1-\cos^2A)\cos^2A\bigr)$ \\ \hline
 $X_{378}$ & $\tan A+\sin 2A$ & $(\sin A\cos A)(1+2\cos^2A)/\cos^2A$ \\ \hline
 $X_{563}$ & $\sin A\sin 4A$ & $(\cos A)(1-\cos^2A)(2\cos^2A-1)$ \\ \hline
 $X_{847}$ & $\tan A\sec 2A$ & $(\sin A\cos A)/\bigl((\cos^2 A)(2\cos^2 A-1)\bigr)$ \\ \hline
\end{tabular}
\end{center}

Thus, the computer search is completely explained by the characterization
developed in Sections~\ref{sec:ops} and \ref{sec:universal}.

\section{Parallel Transversals}
\label{sec:direction}

So far we have treated each admissible transversal independently. The next result
shows that only its \emph{direction} matters. As an admissible transversal slides
parallel to itself, the cevians $AG$, $BH$, and $CI$ remain fixed. The reason is
that each corner triangle changes only by a homothety about its corresponding
vertex.

\begin{lemma}\label{lem:cevfixed}
Let $\LL_0$ and $\LL$ be admissible transversals of $\triangle ABC$ with
$\LL\parallel\LL_0$, and let $X$ be any triangle center. If $G_0,H_0,I_0$ and
$G,H,I$ are the $X$-points of the corner triangles of $\LL_0$ and $\LL$
respectively, then, as lines,
\[
   AG=AG_0,\qquad BH=BH_0,\qquad CI=CI_0 .
\]
\end{lemma}

\begin{proof}
Consider the corner triangle $T_A=AEF$ of $\LL$, with $E=\LL\cap CA$ and
$F=\LL\cap AB$, alongside $T_A^{0}=AE_0F_0$ of $\LL_0$. The sidelines
$CA$ and $AB$ meet at $A$, and the parallels $\LL_0,\LL$ cut them; by the
intercept theorem~\cite{Altshiller-Court} the two signed cut ratios coincide,
\[
   \frac{\overline{AE}}{\overline{AE_0}}
   =\frac{\overline{AF}}{\overline{AF_0}}=:\lambda_A ,
\]
and $\lambda_A\neq0$ since $\LL$ passes through no vertex. Hence
$T_A=h_A(T_A^{0})$, where $h_A$ is the homothety of center $A$
and ratio $\lambda_A$, which fixes $A$ and carries $E_0\mapsto E$ on $CA$ and
$F_0\mapsto F$ on $AB$.

Barycentric coordinates are preserved by every affine map, and a homothety is
affine; so the point with prescribed barycentric coordinates in
$T_A^{0}$ is carried by $h_A$ to the point with the \emph{same}
coordinates in $T_A$.
The $X$-point is defined by these barycentric coordinates
\eqref{eq:cf}, so $G=h_A(G_0)$. As $h_A$ has center $A$, the points $A,G_0,G$ are
collinear; that is, $AG=AG_0$. The arguments at $B$ and $C$ are identical, with
homotheties $h_B,h_C$ of centers $B,C$.
\end{proof}

\begin{theorem}\label{thm:direction}
If a triangle center $X$ is concurrent for an admissible transversal $\LL_0$,
then $X$ is concurrent for every admissible transversal parallel to $\LL_0$.
\end{theorem}

\begin{proof}
By Lemma~\ref{lem:cevfixed}, for any admissible $\LL\parallel\LL_0$ the cevians
$AG,BH,CI$ are the very same three lines as for $\LL_0$. If they concur for
$\LL_0$, they concur for $\LL$.
\end{proof}

\begin{corollary}\label{cor:perspfixed}
If $X$ is concurrent for an admissible transversal $\LL_0$ with perspector $P$,
then $X$ is concurrent for every admissible transversal parallel to $\LL_0$, and
the perspector is the same point $P$.
\end{corollary}

\begin{proof}
The three cevians are unchanged as lines (Lemma~\ref{lem:cevfixed}), so their
common point is unchanged.
\end{proof}

Thus each parallel class of admissible transversals determines a fixed triple of
cevians and, when they are concurrent, a fixed perspector.

\section{The Perspector Formula}

When a center $X$ is concurrent for a fixed admissible transversal $\LL$, we can
compute the coordinates of the perspector explicitly.

\begin{theorem}[Perspector Formula]
\label{thm:perspformula}
Let $X$ be concurrent for the admissible transversal
\[
\LL=[u:v:w],
\]
and let $f$ be the barycentric center function of $X$.  Define
\[
s_A=f(|FA|,|AE|,|EF|),\qquad
t_A=f(|AE|,|EF|,|FA|),
\]
\[
s_B=f(|DB|,|BF|,|FD|),\qquad
t_B=f(|BF|,|FD|,|DB|),
\]
\[
s_C=f(|EC|,|CD|,|DE|),\qquad
t_C=f(|CD|,|DE|,|EC|).
\]
Then the perspector of $\triangle ABC$ and $\triangle GHI$ is
\[
P=
\bigl(
s_As_B(w-v):
t_At_B(u-w):
s_At_B(u-v)
\bigr).
\]
\end{theorem}

\begin{proof}
From the proof of Theorem~\ref{thm:crit},
\[
AG=
\left[
0:
\frac{s_A}{u-w}:
\frac{t_A}{v-u}
\right],
\qquad
BH=
\left[
\frac{t_B}{w-v}:
0:
\frac{s_B}{v-u}
\right].
\]
Their intersection is the cross product of these two line-coordinate vectors.
After clearing the common factor
\[
(u-w)(v-u)(w-v),
\]
we obtain
\[
P=
\bigl(
s_As_B(w-v):
t_At_B(u-w):
s_At_B(u-v)
\bigr).
\]

It remains only to check that this point also lies on $CI$.  Since
\[
CI=
\left[
\frac{s_C}{w-v}:
\frac{t_C}{u-w}:
0
\right],
\]
substitution gives
\[
\frac{s_C}{w-v}\,s_As_B(w-v)
+
\frac{t_C}{u-w}\,t_At_B(u-w)
=
s_As_Bs_C+t_At_Bt_C
=
S+T.
\]
This is zero by Theorem~\ref{thm:crit}, because $X$ is concurrent for $\LL$.  Hence $P$
lies on all three cevians, and is the perspector.
\end{proof}

\section{Expressions for the side lengths}

It is useful to have expressions for the lengths of the sides of the corner triangles.
They are given by the following theorem.

\begin{theorem}
\label{thm:cornerSideLengths}
Let $\LL=[u:v:w]$ be admissible.
Then
\begin{equation}
\label{eq:cornerSideLengths}
\begin{aligned}
FA&=c\left|\frac{u}{u-v}\right|,
&
AE&=b\left|\frac{u}{u-w}\right|,
&
EF&=Q\left|\frac{u}{(u-v)(u-w)}\right|,\\
DB&=a\left|\frac{v}{v-w}\right|,
&
BF&=c\left|\frac{v}{v-u}\right|,
&
FD&=Q\left|\frac{v}{(v-w)(v-u)}\right|,\\
EC&=b\left|\frac{w}{w-u}\right|,
&
CD&=a\left|\frac{w}{w-v}\right|,
&
DE&=Q\left|\frac{w}{(w-u)(w-v)}\right|,
\end{aligned}
\end{equation}
where
$Q=a^2u^2+b^2v^2+c^2w^2-S_Avw-S_Bwu-S_Cuv$,\\
$S_A=b^2+c^2-a^2$, $S_B=c^2+a^2-b^2$, and $S_C=a^2+b^2-c^2$.

\end{theorem}

\begin{proof}
The pierce points are $D=(0:w:-v)$, $E=(-w:0:u)$, $F=(v:-u:0)$.
The result then follows from
direct computation with the barycentric distance formula
$|UV|^2=-(a^2yz+b^2zx+c^2xy)$ applied to the normalized displacement
$(x,y,z)=U-V$.
\end{proof}

Note that the squares of the side lengths are all rational functions of $a,b,c,u,v,w$.

Combining the Perspector Formula with \eqref{eq:cornerSideLengths} gives the
coordinates of the perspector whenever the center function and the equation of
the transversal are known.

\section{The Euler Line}

In earlier sections, we have determined when a center is concurrent for all admissible transversals.
We can also ask when a center is concurrent for some specific transversal $\LL$. We have the following.

\begin{theorem}
There is no admissible transversal $\LL$ such that $X_n$ is concurrent for $\LL$
when $n\in\{1,2,6\}$. 
\end{theorem}

\begin{proof}
Suppose $\LL=[u,v,w]$.
The condition for $X_n$ to be concurrent for $\LL$ is given by Theorem~\ref{thm:crit},
namely $S+T=0$, where $S=s_As_Bs_C$ and $T=t_At_Bt_C$.
We can compute $S$ and $T$ from \eqref{eq:s} and \eqref{eq:cornerSideLengths}.

For $X_1$, $f(a,b,c)=a$ and we find
$$S+T=\frac{2abc|u|\cdot|v|\cdot|w|}{|u-v|\cdot|v-w|\cdot|w-u|}.$$
For an admissible transversal, $u\neq 0$, $v\neq 0$, $w\neq 0$, $u\neq v$, $v\neq w$, and $w\neq u$.
Thus, each term is positive, so $S+T$ cannot equal 0.

For $X_2$, $f(a,b,c)=1$ and we find
$S+T=2$ which is not 0.

For $X_6$, $f(a,b,c)=a^2$ and we find
$$S+T=\frac{2a^2b^2c^2u^2v^2w^2}{(u-v)^2(v-w)^2(w-u)^2}.$$
Again, each term is positive, so $S+T$ cannot equal 0.
\end{proof}

The situation is different for $X_3$ and $X_4$.

By Theorem~\ref{thm:persp}, we know that $X_3$ and $X_4$ are concurrent for \emph{all} admissible lines.
For $n>4$, we can ask if $X_n$ is concurrent for \emph{some} admissible line.
We find for $n=99$, that that although $X_{99}$ is not concurrent for all admissible lines,
it \emph{is} concurrent for the Euler line of $\triangle ABC$.

In fact, we have the following theorem (discovered empirically).

\begin{theorem}[Centers concurrent for the Euler line]
\label{thm:Euler}
Let $\LL$ be the Euler line of $\triangle ABC$.
Three points $D$, $E$, and $F$ are determined by the points
where line $\LL$ intersects the sidelines $BC$, $CA$, and $AB$, respectively.
Let the $X_n$-points of triangles $AEF$, $BFD$, and $CDE$ be $G$, $H$, and $I$ respectively.
Assume that $\LL$ is an admissible transversal.
If
$$n\in\{99, 107, 110, 125, 131, 136, 163, 250, 339, 403, 476, 661, 662,$$
$$669, 670, 798, 799, 822, 823, 850, 925, 930, 974\},$$
then triangles $ABC$ and $GHI$ are perspective.
That is, $X_n$ is concurrent for the Euler line.
\end{theorem}

\begin{remark}
The values of $n$ listed in Theorem~\ref{thm:Euler} are ones that are in addition to the values of $n$
in the general case (Theorem~\ref{thm:persp}), which are also concurrent for the Euler line.
\end{remark}

Before proving Theorem~\ref{thm:Euler}, we give some preliminary results.

Since the Euler line is a fixed transversal, the perspector will be a fixed triangle center.
The equation of the Euler line is given in \cite[\S11.4.4]{Yiu} as
$$S_A(S_B-S_C)x+S_B(S_C-S_A)y+S_C(S_A-S_B)z=0.$$
For a center that is concurrent for the Euler line, we can then use the Perspector Formula
(Theorem~\ref{thm:perspformula}) to find the coordinates for the perspector.
The perspectors for various such centers are listed in the following table.

For centers concurrent for the Euler line that are not in the table,
the perspectors are not cataloged in \cite{ETC} as of January, 2026.

\newpage

\begin{center}
\begin{tabular}{|l|l|}
\hline
\multicolumn{2}{|c|}{\textbf{\color{blue}\large \strut Perspectors for Centers}}\\
\multicolumn{2}{|c|}{\textbf{\color{blue}\large \strut Concurrent for the}}\\
\multicolumn{2}{|c|}{\textbf{\color{blue}\large \strut Euler Line}}\\ \hline
\textbf{center}&\textbf{perspector}\\ \hline
 $X_{3}$ &  $X_{110}$\\ \hline
 $X_{4}$ &  $X_{523}$\\ \hline
 $X_{24}$ &  $X_{15453}$\\ \hline
 $X_{68}$ &  $X_{7471}$\\ \hline
 $X_{69}$ &  $X_{3233}$\\ \hline
 $X_{93}$ &  $X_{18039}$\\ \hline
 $X_{99}$ &  $X_{16163}$\\ \hline
 $X_{107}$ &  $X_{125}$\\ \hline
 $X_{110}$ &  $X_{3}$\\ \hline
 $X_{125}$ &  $X_{107}$\\ \hline
 $X_{265}$ &  $X_{476}$\\ \hline
 $X_{476}$ &  $X_{265}$\\ \hline
 $X_{523}$ &  $X_{4}$\\ \hline
 $X_{850}$ &  $X_{34334}$\\ \hline
\end{tabular}
\end{center}

By Theorem~\ref{thm:direction}, concurrence depends only on the direction of the
transversal, so each center listed in Theorem~\ref{thm:Euler} is concurrent not
merely for the Euler line but for \emph{every} admissible transversal parallel to
it; and by Corollary~\ref{cor:perspfixed} the perspector is, in each case, the
fixed center tabulated above.

Examining the perspectors in the above table suggests that
for the Euler line, center $X$ has perspector $Y$ if and only if center $Y$ has perspector $X$.
We will prove this as Theorem~\ref{thm:reciprocity} below.

The key is the following well-known property of the Euler Line \cite{knot}.

\begin{lemma}\label{lem:cornereuler}
If the Euler line of non-equilateral triangle $ABC$ meets $CA$ at $E$ and meets $AB$ at $F$ (Figure~\ref{fig:EulerLemma}),
and if $\angle A\notin\{60^\circ,120^\circ\}$,
then triangle $AEF$ is not equilateral and its Euler line is parallel to $BC$.
\end{lemma}

\begin{figure}[h!t]
\centering
\includegraphics[width=0.5\linewidth]{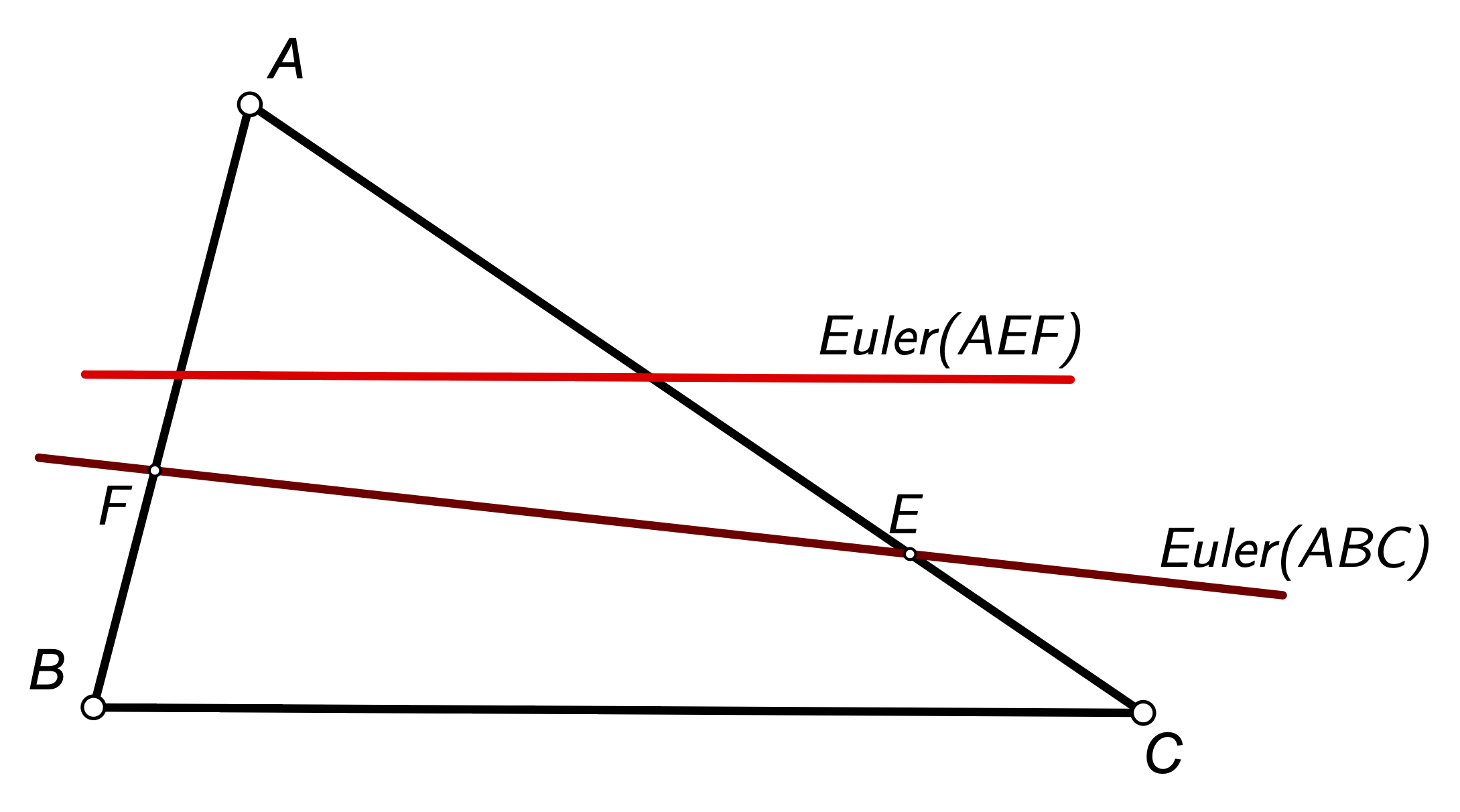}
\caption{red line is parallel to $BC$}
\label{fig:EulerLemma}
\end{figure}

\begin{remark}\label{rem:sixty}
When angle $A$ is $60^\circ$ or $120^\circ$, the conclusion fails for a
picturesque reason: there the corner triangle cut off by the Euler line
is \emph{equilateral}, so it has no Euler line at all.
\end{remark}

\begin{remark}
This lemma propagates to the whole pencil of lines parallel to the Euler line because
parallel transversals cut off corner triangles similar to one another (they are images
of each other under homotheties centered at the vertex).
\end{remark}

\begin{remark}
Similarly, if $\angle B\notin\{60^\circ,120^\circ\}$ the Euler line of $BFD$ is parallel
to $CA$, and if $\angle C\notin\{60^\circ,120^\circ\}$ the Euler line of $CDE$ is
parallel to $AB$.
\end{remark}


\begin{theorem}[The Reciprocity Theorem]
\label{thm:reciprocity}
Suppose no angle of $\triangle ABC$ equals $60^\circ$ or $120^\circ$ and the Euler line
of $\triangle ABC$ is an admissible transversal. Let $X$ be a triangle center that is
concurrent for the Euler line, and let $Y$ be its perspector.
Then $Y$ is concurrent for the Euler line of $\triangle ABC$ with perspector $X$.
\end{theorem}

Simply put, for the Euler line, center $X$
has perspector $Y$ if and only if center $Y$ has perspector $X$.

\begin{proof}
We will prove the result for any line parallel to the Euler line.
Fix an admissible transversal $\LL$ parallel to the Euler line, with pierce points
$D,E,F$, and consider the complete quadrilateral formed by the four lines
$BC$, $CA$, $AB$,~$\LL$ (see Figure~\ref{fig:recip1}).
Its four triangles are $ABC$, $AEF$, $BFD$, and $CDE$.
Each omits
exactly one of the four lines, and the omitted line is a transversal of it. By
Lemma~\ref{lem:cornereuler} (and the remarks following), the configuration is \emph{self-reciprocal}: each of the
three corner triangles sees its omitted line as a transversal parallel to its own Euler
line, namely
\[
  BC\parallel \text{Euler}(AEF),\qquad
  CA\parallel \text{Euler}(BFD),\qquad
  AB\parallel \text{Euler}(CDE),
\]
and these transversals are admissible for the corner triangles (their pierce points are
$D,B,C$ on the sidelines of $AEF$, and cyclically, none of which is a vertex).

Now consider the center $Y$ in the corner triangle $AEF$.
By Corollary~\ref{cor:perspfixed}, the $Y$-point of $AEF$ is the perspector formed by taking $X$-points
in the corner triangles formed by triangle $AEF$ and \emph{any} admissible transversal
parallel to the Euler line of $\triangle AEF$.
By the self-reciprocity just established, we may
take that transversal to be the line $BC$. The transversal $BC$ meets the sidelines
$EF$, $FA$, $AE$ of $\triangle AEF$ at $D$, $B$, $C$ respectively, so the corner
triangles of the pair $(\triangle AEF,\,BC)$ at the vertices $A$, $E$, $F$ are
\[
   ABC,\qquad CDE,\qquad BFD,
\]
and the corresponding corner centers (using $X$-points) are $X=X_{ABC}$ itself,
$X_{CDE}$, and~$X_{BFD}$ (Figure~\ref{fig:recip1}).
In particular, the cevian issuing from $A$ is the
line $AX$, and the perspector, which is the $Y$-point of $\triangle AEF$, lies on it. Hence,
when $Y$-points are taken in the corner triangles formed by $\triangle ABC$ and $\LL$,
the line joining $A$ to the $Y$-point of $\triangle AEF$ is the line $AX$ and passes through $X$.

\begin{figure}[h!t]
\centering
\includegraphics[width=1\linewidth]{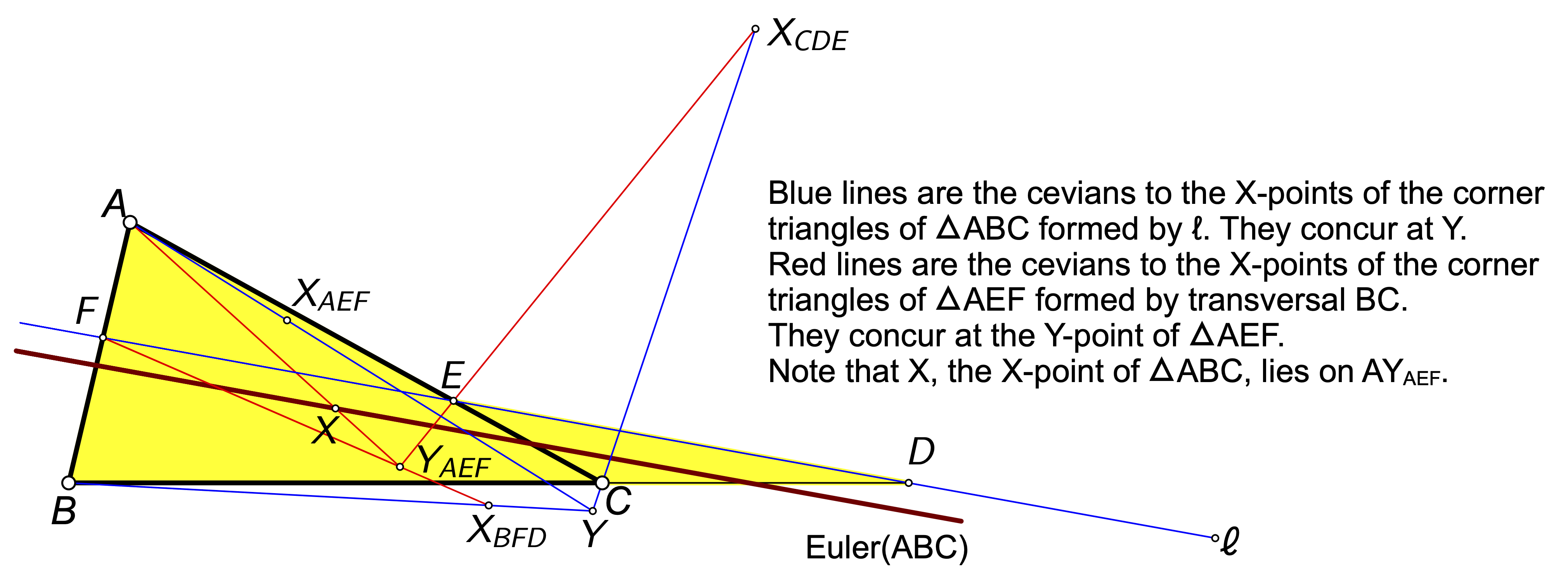}
\caption{construction of $Y$-point of $\triangle AEF$}
\label{fig:recip1}
\end{figure}

The same argument applied to the pairs $(\triangle BFD,\,CA)$ and $(\triangle CDE,\,AB)$.
Their corner triangles at $B$ and at $C$ are again $ABC$. This shows that the cevians
from $B$ and from $C$  (when using $Y$-points in the corner triangles) are the lines $BX$ and $CX$ (Figure~\ref{fig:recip2}).
\begin{figure}[h!t]
\centering
\includegraphics[width=1\linewidth]{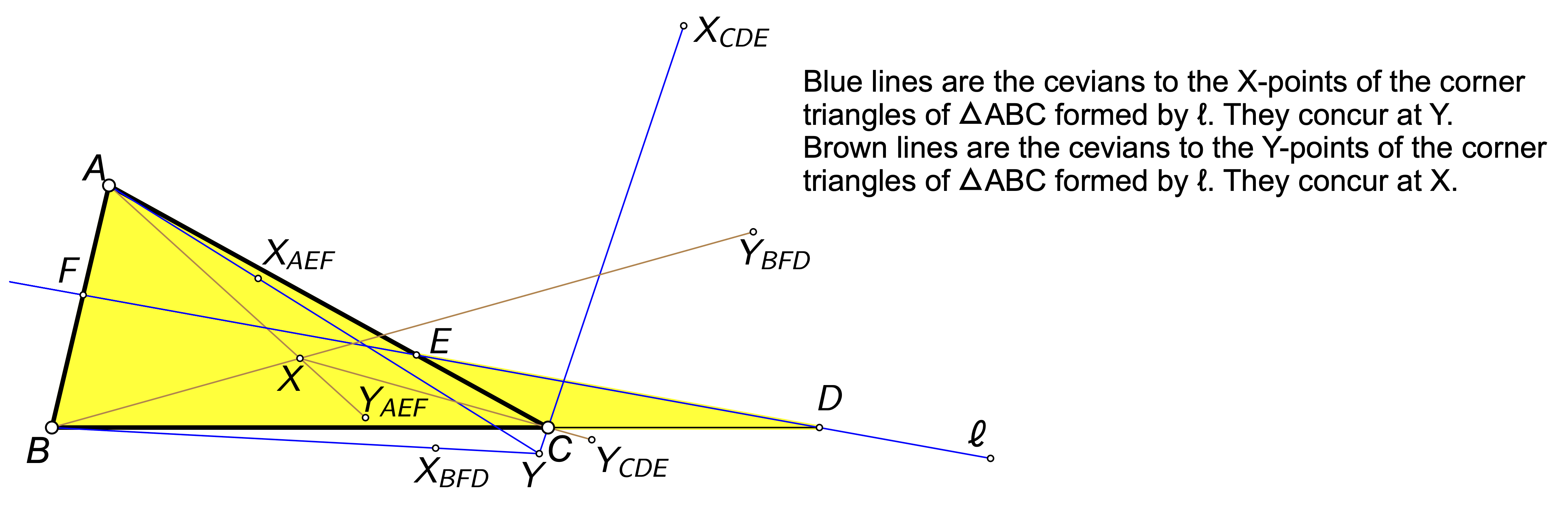}
\caption{construction of $X$ and $Y$ points of $\triangle ABC$}
\label{fig:recip2}
\end{figure}

The three cevians to the $Y$-points of the corner triangles of $\triangle ABC$
therefore concur at $X$, showing that the center $Y$ is concurrent for
$\LL$, with perspector~$X$.
\end{proof}

\goodbreak
\begin{proposition}\label{prop:steiner-euler}
The Steiner point, $X_{99}$, is concurrent for the Euler line.
\end{proposition}

\begin{proof}
The Steiner point has barycentric center function $f(x,y,z)=1/(y^{2}-z^{2})$.
Put
\[
  p=|AE|^{2}-|EF|^{2},\qquad p'=|EF|^{2}-|FA|^{2},
\]
and let $q,q',r,r'$ be the cyclic analogues. By
Theorem~\ref{thm:crit}, $s_A=1/p$, $t_A=1/p'$, and so on, whence
\[
  S+T=\frac{p'q'r'+pqr}{pqr\,p'q'r'},
\]
so $AG,BH,CI$ concur if and only if
$pqr+p'q'r'=0$.

The trace points are $D=(0:w:-v)$, $E=(-w:0:u)$, $F=(v:-u:0)$. Using the
barycentric distance formula, one computes
\[
  p=\frac{-u^{2}m_A}{(u-v)^{2}(u-w)},\qquad
  p'=\frac{u^{2}m_A'}{(u-v)(u-w)^{2}},
\]
where
\[
  m_A=(a^{2}-b^{2})(u-v)+c^{2}(v-w),\qquad
  m_A'=(a^{2}-c^{2})(u-w)+b^{2}(w-v),
\]
the remaining quantities following by the cyclic substitution
$(a,b,c,u,v,w)\mapsto(b,c,a,v,w,u)$. Multiplying and using
$(u-w)(v-u)(w-v)=-(u-v)(v-w)(w-u)$,
we have
\[
  pqr=\frac{u^{2}v^{2}w^{2}\,m_Am_Bm_C}{\Delta^{3}},\qquad
  p'q'r'=\frac{u^{2}v^{2}w^{2}\,m_A'm_B'm_C'}{\Delta^{3}},
\]
where $\Delta=(u-v)(v-w)(w-u)$.

As $u,v,w$ are nonzero and distinct, concurrency is equivalent to
\[
  \nu:=m_Am_Bm_C+m_A'm_B'm_C'=0.
\]

The cubic $\nu$ factors as $\nu=-L\,Q$ with $L$ linear and $Q$ quadratic in
$(u,v,w)$. The coefficient vector of $L$ is
\[
  \bigl(\,2a^{4}-a^{2}b^{2}-a^{2}c^{2}-b^{4}+2b^{2}c^{2}-c^{4}\;:\ :\,\bigr)
  \;=\;X_4-X_3,
\]
the orthocenter minus the circumcenter, taken in the common normalization in
which each has coordinate sum $16K^{2}$ (with $K$ the area of $ABC$). Its
coordinates sum to zero, so $X_4-X_3$ is the point at infinity of the line
$X_3X_4$, namely of the Euler line. Hence $L=0$ expresses exactly that
$\LL$ passes through the infinite point of the Euler line, i.e.\ that $\LL$
is parallel to it. By hypothesis this holds, so $\nu=0$; therefore $S+T=0$,
and $AG,BH,CI$ concur.
\end{proof}

If $X$ is a concurrent center, by \emph{the orbit of $X$} we mean the set of all
centers that can be obtained from $X$ by successively applying the preserving operations (O1)-(O6)
of Definition~\ref{def:ops}.

\begin{theorem}\label{thm:orbit}
Every center in the orbit of $X_{99}$ is concurrent for the Euler line.
\end{theorem}

\begin{proof}
Since $X_{99}$ is concurrent for the Euler line, the result follows from Theorem~\ref{thm:pres}.
\end{proof}

\begin{theorem}\label{cor:members}
The center $X_n$ is concurrent for the Euler line for\\
$n\in\{107,110,163,661,662,669,670,798,799,822,823,850,930\}$.
\end{theorem}

\begin{proof}
From \cite{ETC}, we find that the center function for $X_{99}$, the Steiner point,
is $(a^2-b^2)(c^2-a^2)$ and that the isogonal conjugate of $X_{99}$ is $X_{512}$.
The following identities therefore show that each $X_n$ is concurrent by Theorem~\ref{thm:orbit}.
\[
\begin{array}{llll}
X_{512}=a^2/X_{99}, &X_{523}=1/X_{99}&\\[2pt]
X_{107}=S_A^{-2}X_{99}, &X_{110}=a^{2}X_{99}, &  X_{163}=a^{3}X_{99}, &\\[2pt]
X_{662}=aX_{99}, & X_{669}=a^4X_{512}, & X_{670}=(bc)^{2}X_{99}, & \\[2pt]
X_{661}=a^2/X_{662}, & X_{799}=bc\,X_{99}, & X_{822}=a^{3}S_A^{2}X_{512}, &\\[2pt]
X_{823}=bc(S_BS_C)^{2}X_{99}, & X_{850}=(bc)^{2}X_{512}, &X_{930} = (bc)^{-2}\,\bigl(4\cos^{2}\!A - 3\bigr)^{-1} X_{99}, &\\[2pt]
X_{798}=a^3/X_{99}, & &X_{925} = (bc)^{-2}\,\bigl(4\cos^{2}\!A - 2\bigr)^{-1} X_{99}.
\end{array}
\]
Note: For $X_{930}$, $a^4+b^4+c^4-2a^2b^2-2a^2c^2-b^2c^2 = S_A^{2}-3b^2c^2 = (bc)^2(4\cos^2\!A-3)$,
so dividing by it is the co-side-power operation (O3) with $j=-2$ followed by the
angular multiplier (O6) with $\Phi(t)=1/(4t-3)$; both preserve concurrence, so
$X_{930}$ lies in the orbit of $X_{99}$. A similar remark holds for $X_{925}$.
\end{proof}

\begin{theorem}
The center $X_n$ is concurrent for the Euler line for\\
$n\in\{125, 250, 339, 476\}$.
\end{theorem}

\begin{proof}
Since $X_{107}$ is concurrent for the Euler line with perspector $X_{125}$,
by the Reciprocity Theorem (Theorem~\ref{thm:reciprocity}), $X_{125}$ must be concurrent for the Euler line.
Similarly, $X_{265}$ is concurrent for all lines (Theorem~\ref{thm:persp}), and for the Euler line,
the perspector is $X_{476}$. Hence $X_{476}$ must be concurrent for the Euler line.
Examining their center functions from \cite{ETC}, we find
\[
\begin{array}{llll}
X_{250}=a^2/X_{125}, &X_{339}=1/X_{250}, 
\end{array}
\]
so $X_{250}$ and $X_{339}$ are concurrent for the Euler line by Theorem~\ref{thm:pres}.
\end{proof}

\goodbreak

\begin{theorem}\label{thm:x131}
The center $X_{131}$ is concurrent for the Euler line.
\end{theorem}

\begin{proof}
By Theorem~\ref{thm:crit}, concurrence for a given transversal is
equivalent to the identity $S+T=0$. Substituting the ETC center function of
$X_{131}$, the side-length expressions of Theorem~\ref{thm:cornerSideLengths}, and the
coordinates of the Euler line reduces $S+T=0$ to a polynomial identity in
$a,b,c$. We have verified this identity with Mathematica. The
intermediate polynomials are large and unilluminating, so we omit them.
\end{proof}

\begin{theorem}
The center $X_n$ is concurrent for the Euler line for\\
$n\in\{136, 403, 974\}$.
\end{theorem}

Concurrence for these centers is verified by the same computer-algebra reduction
as in Theorem~\ref{thm:x131}, and the details are omitted.

We leave the reader with something to ponder on.

\begin{conjecture}
\label{conj:complete}
A center on the circumcircle is concurrent for the Euler line if and
only if it lies in the orbit of $X_{99}$.
\end{conjecture}

\section*{Acknowledgments}

The authors gratefully acknowledge the Anthropic AI program Claude (Opus 4.8 and Fable 5) for substantial assistance
in proving the results of this paper.
We also thank OpenAI's ChatGPT (version 5.5) for clarifying and simplifying much of the exposition.


\end{document}